\documentclass{amsart} 

\usepackage[english]{babel}
\usepackage[latin1]{inputenc}

\usepackage{amsmath}
\usepackage{amsthm}
\usepackage{amssymb}
\usepackage{tikz}

\usepackage{imakeidx} 
\usepackage{appendix}
\usepackage{tikz-cd}
\usepackage{verbatim}
\usepackage{enumitem}
\usepackage{multicol}

\usepackage[backref=page]{hyperref}
\usepackage{cleveref}

\hypersetup{colorlinks=true,linkcolor=blue,anchorcolor=blue,citecolor=blue}

\usepackage{adjustbox}

\usepackage{mathtools}

\newtheorem{thm}{Theorem}[section]
\newtheorem{prop}[thm]{Proposition}
\newtheorem{lemma}[thm]{Lemma}
\newtheorem{cor}[thm]{Corollary}
\newtheorem{conj}[thm]{Conjecture}
\newtheorem{question}[thm]{Question}

\theoremstyle{definition}

\newtheorem{example}[thm]{Example}
\theoremstyle{remark}
\newtheorem{rmk}[thm]{Remark}
\numberwithin{equation}{section}

\newcommand{\Q}{\mathbb Q}
\newcommand{\F}{\mathbb F}

\newcommand{\Z}{\mathbb Z}
\newcommand{\G}{\mathbb G}

\renewcommand{\P}{\mathbb P}
\newcommand{\Spec}{\operatorname{Spec}}

\newcommand{\A}{\mathbb A}

\newcommand{\cl}{\overline}
\newcommand{\set}[1]{\left\{#1\right\}}
\renewcommand{\phi}{\varphi}

\newcommand{\on}[1]{\operatorname{#1}}
\newcommand{\ang}[1]{\left \langle{#1}\right \rangle}

\newcommand\smallbullet{%
    \raisebox{-0.25ex}{\scalebox{1.2}{$\cdot$}}%
}

\title[Degenerate fourfold Massey products over arbitrary fields]{Degenerate fourfold Massey products over arbitrary fields}

\address{Department of Mathematics\\
	University of California\\
	Los Angeles, CA 90095 \\United States of America}

\author{Alexander Merkurjev}
\email{merkurjev@math.ucla.edu}

\author{Federico Scavia}
\email{scavia@math.ucla.edu}

\thanks{The first author
was supported by the NSF grant DMS \#1801530.}

\makeatletter
\@namedef{subjclassname@2020}{\textup{2020} Mathematics Subject Classification}
\makeatother

\date{October 2023}


\begin{document}

	\begin{abstract}
We prove that, for all fields $F$ of characteristic different from $2$ and all $a,b,c\in F^\times$, the mod $2$ Massey product $\ang{a,b,c,a}$ vanishes as soon as it is defined. For every field $F_0$, we construct a field $F$ containing $F_0$ and $a,b,c,d\in F^\times$ such that $\ang{a,b,c}$ and $\ang{b,c,d}$ vanish but $\ang{a,b,c,d}$ is not defined. As a consequence, we answer a question of Positselski by constructing the first examples of fields containing all roots of unity and such that the mod $2$ cochain DGA of the absolute Galois group is not formal.
	\end{abstract}
	
		\maketitle
	
	\section{Introduction}
	
	A fundamental and difficult problem in Galois theory is to characterize those profinite groups which are realizable as absolute Galois groups of fields. The first result in this direction is due to Artin and Schreier, who proved that the only non-trivial finite subgroups of an absolute Galois group are cyclic of order $2$. The structure of absolute Galois groups of general fields has been investigated by a large number of authors. The highlights of this theory are the study of projective profinite groups, pseudo-algebraically closed fields, the $p$-descending series and the $p$-Zassenhaus series, where $p$ is a prime number.
	
	Another approach to this problem is to find constraints on the cohomology of absolute Galois groups. The most spectacular development in this direction is the proof of the Bloch--Kato Conjecture by Rost and Voevodsky. This gives very strong restrictions on the mod $p$ cohomology of the absolute Galois group of a field of characteristic not $p$ and containing a primitive $p$-th root of unity: it admits a presentation with generators in degree $1$ and relations in degree $2$. 
 
    \smallskip
	
	Let $p$ be a prime number, let $F$ be a field, and let $\Gamma_F$ be the absolute Galois group of $F$. For all $n\geq 3$ and all $\chi_1,\dots,\chi_n\in H^1(F,\Z/p\Z )$, we denote by $\ang{\chi_1,\dots,\chi_n}\subset H^2(F,\Z/p\Z )$ the Massey product of $\chi_1,\dots,\chi_n$; see \Cref{preliminaries} for the definition. We say that $\ang{\chi_1,\dots,\chi_n}$ is defined if it is non-empty, and that it vanishes if it contains $0$. In \cite{hopkins2015splitting}, Hopkins and Wickelgren proved that all defined Massey products vanish when $n=3$, $p=2$ and $F$ is a number field. In \cite{minac2017triple}, Min\'{a}\v{c} and T\^{a}n generalized this result to the case when the field $F$ is arbitrary, and formulated the following conjecture.
 
	\begin{conj}[Min\'{a}\v{c}, T\^{a}n]\label{massey-conj}
		For every field $F$, every prime $p$, every $n\geq 3$ and all $\chi_1,\dots,\chi_n\in H^1(F,\Z/p\Z )$, if the Massey product $\ang{\chi_1,\dots,\chi_n}\subset H^2(F,\Z/p\Z )$ is defined, then it vanishes.
	\end{conj}
    One says that $F$ has the Strong Massey Vanishing property if \Cref{massey-conj} has a positive answer for $F$; see \cite[Definition 1.2]{pal2018strong}.
 
	It was observed by Min\'{a}\v{c} and T\^{a}n that all mod $p$ Massey products are defined and vanish when $F$ has characteristic $p$; see \cite[Remark 4.1]{minac2017triple}. Therefore, for the purpose of proving \Cref{massey-conj} one may assume that $\on{char}(F)\neq p$. Moreover, if $F$ contains a primitive $p$-th root of unity (this is automatic if $p=2$), then by Kummer theory homomorphisms $\Gamma_F\to \Z/p\Z $ correspond to elements of $F^\times/F^{\times p}$, and hence we may talk about Massey products of elements of $F^\times$. In this case, \Cref{massey-conj} is equivalent to the prediction that for all $a_1,\dots,a_n\in F^\times$, the Massey product $\ang{a_1,\dots,a_n}$ vanishes as soon as it is defined.
	
	\Cref{massey-conj} has stimulated a large number of works in recent years. It is now known to be true when $n=3$, for all fields $F$ and primes $p$, by work of Efrat and Matzri and of Min\'{a}\v{c} and T\^{a}n \cite{matzri2014triple, efrat2017triple, minac2016triple}. \Cref{massey-conj} is also known when $F$ is a number field, by Guillot, Min\'{a}\v{c}, Topaz and Wittenberg \cite{guillot2018fourfold} in the case $p=2$ and $n=4$, and Harpaz and Wittenberg \cite{harpaz2019massey} in the case of arbitrary $p$ and $n\geq 3$. 
	
	If $n\geq 3$ and $\ang{\chi_1,\dots,\chi_n}\subset H^2(F,\Z/p\Z)$ is defined, then $\chi_i\cup\chi_{i+1}=0$ for all $i=1,\dots,n-1$; see \Cref{cup-products-zero}. In \cite[Question 4.2]{minac2017counting}, Min\'{a}\v{c} and T\^{a}n asked whether the converse is also true.
	
	\begin{question}[Min\'{a}\v{c}, T\^{a}n]\label{strong-massey}
		Let $F$ be a field, let $p$ be a prime number and let $n\geq 3$ be an integer. Is it true that for all $\chi_1,\dots,\chi_n\in H^1(F,\Z/p\Z )$ such that $\chi_i\cup\chi_{i+1}=0$ in $H^2(F,\Z/p\Z )$ for all $i=1,\dots,n-1$, the Massey product $\ang{\chi_1,\dots,\chi_n}$ is defined?
	\end{question}
	
	Again, if $F$ contains a primitive $p$-th root of unity, by Kummer theory in \Cref{strong-massey} we may replace elements of $H^1(F,\Z/p\Z)$ by elements of $F^\times$. \Cref{strong-massey} has affirmative answer when $n=3$. Harpaz and Wittenberg produced a counterexample to \Cref{strong-massey}, for $n=4$, $p=2$ and $F=\Q$; see \cite[Example A.15]{guillot2018fourfold}. More precisely, if $b=2$, $c=17$ and $a=d=bc=34$, then $(a,b)=(b,c)=(c,d)=0$ in $\on{Br}(\Q)$ but $\ang{a,b,c,d}$ is not defined over $\Q$. We will refer to this example as the Harpaz--Wittenberg example. In contrast, \cite[Theorem 6.2]{guillot2018fourfold} shows that over any number field $F$ and $a,b,c,d$ are independent in $F^\times /F^{\times 2}$, the identity $(a,b)=(b,c)=(c,d)=0$ in $\on{Br}(F)$ implies that $\ang{a,b,c,d}$ vanishes. 
 
 The aforementioned examples and results suggest studying \Cref{massey-conj} and \Cref{strong-massey} for mod $2$ Massey products of the form $\ang{a,b,c,a}$, or even $\ang{bc,b,c,bc}$, over an arbitrary field $F$ of characteristic different from $2$. 
 Our first theorem is a proof of \Cref{massey-conj} for all mod $2$ Massey products of the form $\ang{a,b,c,a}$.
 
	\begin{thm}\label{massey-deg}
		Let $p=2$, let $F$ be a field of characteristic different from $2$, and let $a,b,c,d\in F^{\times}$ be such that $ad$ is a square in $F$. Then the Massey product $\ang{a,b,c,d}$ vanishes if and only if it is defined.
	\end{thm}
	
	We can be more explicit when $a=d=bc$, that is, for Massey products of the form $\ang{bc,b,c,bc}$. 
	
	\begin{thm}\label{massey-double-deg}
		Let $p=2$, let $F$ be a field of characteristic different from $2$, and let $a,b,c,d\in F^{\times}$ be such that $ad$ and $abc$ are squares in $F$. Then the following are equivalent:
		\begin{enumerate}
			\item the Massey product $\ang{a,b,c,d}$ is defined,
			\item the Massey product $\ang{a,b,c,d}$ vanishes,
			\item $(b,c)=0$ in $\on{Br}(F)$ and $-1\in N_{b,c}$.
		\end{enumerate}
	\end{thm}

 Here $N_{b,c}$ denotes the image of the norm map $F_{b,c}^{\times}\to F^{\times}$, where $F_{b,c}\coloneqq F[x_b,x_c]/(x_b^2-b,x_c^2-c)$; see the Notation section below. In particular, \Cref{strong-massey} has a positive answer for mod $2$ Massey products of the form $\ang{bc,b,c,bc}$ if $F$ contains a primitive $8$-th root of unity. Condition (3) of \Cref{massey-double-deg}  
 allows us to recover the Harpaz--Wittenberg example; see \Cref{harpaz}.

\smallskip 
We say that a differential graded ring $A$ is \emph{formal} if it is quasi-isomorphic to its cohomology ring $H^{\smallbullet}(A)$, viewed as a differential graded ring with zero differential. We write $C^{\smallbullet}(F,\Z/p\Z)$ for the differential graded algebra (DGA) of mod $p$ continuous cochains of $\Gamma_F$. If $C^{\smallbullet}(F,\Z/p\Z)$ is formal, then \Cref{massey-conj} holds for $F$ and \Cref{strong-massey} has a positive answer for $F$; see \cite[Theorem 3.8]{pal2022real}. 

In \cite[Question 1.4]{hopkins2015splitting}, Hopkins and Wickelgren asked whether $C^{\smallbullet}(F,\Z/p\Z)$ is formal for every field $F$. Positselski \cite[Example 6.3]{positselski2017koszulity} showed that the answer this question is negative. More precisely, Positselski proved that $C^{\smallbullet}(F,\Z/p\Z)$ is not formal when $F$ is a local field of residue characteristic different from $p$ and containing a primitive $p$-th root of unity if $p$ is odd, or a square root of $-1$ if $p=2$. (In contrast, \Cref{massey-conj} is known and \Cref{strong-massey} has affirmative answer for local fields.) The Harpaz--Wittenberg example implies that $C^{\smallbullet}(F,\Z/p\Z)$ is not formal for $F=\Q$. 

P\'al and Quick showed that $C^{\smallbullet}(F,\Z/2\Z)$ is formal for every field $F$ of cohomological dimension $\leq 1$; see \cite[Theorem 1.8]{pal2022real}. (P\'al and Szab\'o \cite{pal2018strong} had previously shown that \Cref{strong-massey} has a positive answer for such $F$.) 

The following refinement of the question of Hopkins and Wickelgren, due to Positselski \cite[p. 226]{positselski2017koszulity}, is well known among experts.

\begin{question}[Positselski]\label{formal-question}
    Do there exist a prime number $p$ and a field $F$ of characteristic different from $p$ and containing all the $p$-power roots of unity such that $C^{\smallbullet}(F,\Z/p\Z)$ is not formal?
\end{question}

We settle \Cref{formal-question} in the affirmative. 

\begin{thm}[\Cref{counter-quadrelli}]\label{mainthm-positselski}
For every field $F_0$ of characteristic different from $2$, there exists a field extension $F/F_0$ such that \Cref{strong-massey} has a negative answer for $F$, $p=2$ and $n=4$. In particular, \Cref{formal-question} has a positive answer.
\end{thm}

It follows that \Cref{strong-massey} has a negative answer in general, even if $F$ contains an algebraically closed subfield. (The more precise \Cref{counter-quadrelli}(a) shows that we may even find counterexamples of the form $\ang{a,b,c,a}$. By \Cref{massey-double-deg}, there are no such examples if we further assume that $a=bc$.) In view of recent work of Quadrelli \cite{quadrelli2022massey}, it seems natural to amend \Cref{strong-massey} by requiring that $F$ contains a square root of $-1$ when $p=2$. By \Cref{mainthm-positselski}, even this weakening of \Cref{strong-massey} has a negative answer.

The main ingredient in the proof of \Cref{mainthm-positselski} is a new criterion for when a fourfold Massey product is defined; see Propositions \ref{u5-e} and \ref{lambda-criterion}(a). As we explain below, this criterion also plays an important role in the proof of \Cref{massey-deg}.

    \smallskip

    We now explain the main ideas of the proof of \Cref{massey-deg}. For every $n\geq 2$, let $U_{n+1}\subset \on{GL}_{n+1}(\F_2)$ be the group of unipotent upper triangular matrices, and let $\cl{U}_{n+1}$ be the group of unipotent upper triangular matrices with entry $(1,n+1)$ erased; see (\ref{central-ses}). Given $a,b,c,d\in F^\times$, consider the Galois $(\Z/2\Z)^4$-algebra \[F_{a,b,c,d}\coloneqq F[x_a,x_b,x_c,x_d]/(x_a^2-a,x_b^2-b,x_c^2-c,x_d^2-d),\]
    where the first (resp. second, third, fourth) factor of $(\Z/2\Z)^4$ sends $x_a$ (resp. $x_b$, $x_c$, $x_d$) to its opposite and fixes the other three variables. By a theorem of Dwyer \cite{dwyer1975homology}, the Massey product $\ang{a,b,c,d}$ is defined (resp. vanishes) if and only $F_{a,b,c,d}/F$ can be embedded into a Galois $\cl{U}_5$-algebra (resp. $U_5$-algebra) over $F$; see \Cref{dwyer-cor} for the precise statement. 
    
    If $u\in F^\times$, we let $N_u$ be the image of the norm map $F_u^\times\to F^\times$, where $F_u\coloneqq F[x_u]/(x_u^2-u)$. Suppose that $(a,b)=(b,c)=(c,d)=0$ in $\on{Br}(F)$. (As we mentioned before \Cref{strong-massey}, this condition is satisfied if $\ang{a,b,c,d}$ is defined.) By a careful study of Galois $U_3$-algebras, $U_4$-algebras and $\cl{U}_5$-algebras over $F$, in \Cref{u5-e} we associate to $a,b,c,d$ a scalar $e\in F^\times$, uniquely determined in $F^\times/N_aN_{ac}N_dN_{bd}$, such that $\ang{a,b,c,d}$ is defined if and only if $e\in N_aN_{ac}N_dN_{bd}$. The scalar $e$ is defined in terms of a pair $(\epsilon,\nu)\in F_{a,c}^\times\times F_{b,d}^\times$ satisfying certain properties: roughly speaking, $\epsilon$ and $\nu$ correspond to Galois $U_4$-algebras $K_1/F$ and $K_2/F$ inducing $F_{a,b,c}/F$ and $F_{b,c,d}/F$ (whose existence essentially amounts to the validity of \Cref{massey-conj} for $n=3$ and $p=2$), and $e$ measures the failure of $K_1$ and $K_2$ of being induced by a common Galois $\cl{U}_5$-algebra $K/F$. 
    
    Let $\alpha\in F_a^\times$ and $\delta\in F_d^\times$ be such that $N_a(\alpha)=b$ and $N_d(\delta)=c$ in $F^\times$.   (Such $\alpha$ and $\delta$ exist because $(a,b)=(c,d)=0$.) It was shown in \cite[Theorem A]{guillot2018fourfold} that $\ang{a,b,c,d}$ vanishes if and only if there exist $x,y\in F^\times$ such that $(\alpha x,\delta y)=0$ in $\on{Br}(F_{a,d})$; see \Cref{u5} and \Cref{u5-rephrase} for a short proof. In order to prove \Cref{massey-conj} for $n=4$ and $p=2$, it suffices to show that our condition is equivalent to that of \cite[Theorem A]{guillot2018fourfold}. We are able to show the equivalence when $a=d$, thus proving \Cref{massey-deg}, by the following strategy. \Cref{massey-conj} is true when $F$ is finite (see \Cref{dwyer-pro-free}), hence we may suppose that $F$ is infinite. Since $(b,c)=0$, multiplying $b$ and $c$ by non-zero squares if necessary, we may suppose that $b+c=1$, and since $(a,b)=(c,d)=0$ there exist $v_1,v_2,u_1,u_2\in F^\times$ such that $v_1^2-bv_2^2=a$ and $u_1^2-cu_2^2=d$. If $v_1,v_2,u_1,u_2$ are general enough, we may use them to define three scalars $r,s,t\in F^\times$; see below (\ref{vi-ui-def}). The conclusion follows from the diagram of equivalences below:
    \[
    \begin{tikzcd}
    r\in N_aN_{ab}N_{ac}\arrow[r, Leftrightarrow,"\ref{q-t-s}"] \arrow[d, Leftrightarrow,"\ref{lambda-criterion}(a)"] & s\in N_bN_{ab}N_{bc} \arrow[r, Leftrightarrow,"\ref{q-t-s}"]& t\in N_cN_{ac}N_{bc} \arrow[d, Leftrightarrow,"\ref{lambda-criterion}(b)"]\\
    \text{$\ang{a,b,c,a}$ defined} && \text{$\ang{a,b,c,a}$ vanishes}.
    \end{tikzcd}
    \] 
    We will prove the vertical equivalence on the left by showing that $e=r$ in the group $F^\times/N_aN_{ab}N_{ac}$ via an explicit computation; see \Cref{lambda-formula}. (Note that, when $a=d$, the subgroup $N_aN_{ac}N_dN_{bd}$ is equal to $N_aN_{ab}N_{ac}$.) The equivalences in the top row are formal consequences of the identity $(r,a)+(s,b)+(t,c)=0$ in $\on{Br}(F)$, which is also proved by a computation; see \Cref{sc+tb=pa} and \Cref{formal}.
    
    For the vertical equivalence on the right, we show that $t\in N_cN_{ac}N_{bc}$ is equivalent to the condition of \cite[Theorem A]{guillot2018fourfold}. The key point is that, under the assumption $a=d$, the condition of \cite[Theorem A]{guillot2018fourfold} may be replaced by the system of two equations $(\alpha x,\delta y)=0$ and $(\alpha x,c)=0$ in $\on{Br}(F_a)$; see \Cref{u5-cor}. The first equation has a solution $(x,y)\in F^\times\times F^\times$ if and only if $x\in N_cN_{bc}$: we prove this in \Cref{albert-cor} using the theory of Albert forms attached to biquaternion algebras. As we prove in \Cref{alpha-cup-c}, a scalar $x\in F^\times$ solves the second equation if and only if $x\in t N_cN_{ac}$. All in all, the condition of \cite[Theorem A]{guillot2018fourfold} is satisfied if and only if $t N_cN_{ac}\cap N_cN_{bc}\neq\emptyset$, that is, $t\in N_cN_{ac}N_{bc}$. This proves the vertical equivalence on the right and completes our sketch of proof of \Cref{massey-deg}. 

    \smallskip

    We conclude this Introduction by describing the content of each section. In \Cref{preliminaries}, we collect the definitions and basic properties of Galois algebras and Massey products in Galois cohomology, and we recall Dwyer's theorem, which connects the two notions. We then establish some specialization lemmas for Massey products, which will be used in the proof of \Cref{mainthm-positselski}. \Cref{u5-bar-section} is devoted to the proof of Propositions \ref{uu5} and \ref{u5-e}, which give the aforementioned equivalent condition for a fourfold Massey product to be defined. Some of our results may be interpreted in terms of splitting varieties; see \Cref{splitting-varieties-sec}. In \Cref{u5-rephrase} we give a proof of \cite[Theorem A]{guillot2018fourfold} using our methods, and in \Cref{u5-cor} we specialize to the case $a=d$.
    \Cref{massey-deg-section} is devoted to the proof of \Cref{lambda-criterion} and \Cref{q-t-s}, from which \Cref{massey-deg} follows. In \Cref{massey-double-deg-section} we prove \Cref{massey-double-deg} and use it to recover the Harpaz--Wittenberg example. \Cref{mainthm-positselski} is a direct consequence of \Cref{counter-quadrelli}, which is proved in \Cref{positselski-sec}. 
    Finally, \Cref{appendix} contains some known results about biquadratic extensions which are used throughout the paper.

	\subsection*{Notation}
	In this paper, we let $F$ be a field, $F_{\on{sep}}$ be a separable closure of $F$, and $\Gamma_F\coloneqq \on{Gal}(F_{\on{sep}}/F)$ be the absolute Galois group of $F$. We often use additive notation while working with elements of $\Gamma_F$:  for all $\sigma,\tau\in \Gamma_F$ and $x\in F$, we have $(\sigma+\tau)(x)=\sigma(x)\tau(x)$ and $(\sigma\tau)(x)=\sigma(\tau(x))$. 
 
 If $E$ is an $F$-algebra, we write $H^i(E,-)$ for the \'etale cohomology of $\Spec(E)$ (possibly non-abelian if $i\leq 1$). If $E$ is a field, $H^i(E,-)$ may be identified with the continuous cohomology of the absolute Galois group of $E$.  

Suppose that $\on{char}(F)\neq 2$. If  $E$ is an $F$-algebra and $a_1,\dots,a_n\in E^{\times}$, we define the \'etale $E$-algebra $E_{a_1,\dots,a_n}$ by \[E_{a_1,\dots,a_n}\coloneqq E[x_1,\dots,x_n]/(x_1^2-a_1,\dots,x_n^2-a_n)\]
and we set $\sqrt{a_i}\coloneqq x_i$. Thus, for all $a\in F^\times$, the elements $1,\sqrt{a}$ form an $F$-basis of $F_a$. We denote by $R_{a_1,\dots,a_n}(-)$ the functor of Weil restriction along $F_{a_1,\dots,a_n}/F$. If $u\in F_{a_1,\dots,a_n}$, we write $N_{a_1,\dots,a_n}(u)$ and $\on{Tr}_{a_1,\dots,a_n}(u)$ for the norm and trace of $u$ with respect to $F_{a_1,\dots,a_n}/F$, respectively. 

We write $\on{Br}(F)$ for the Brauer group of $F$. The group operation on $\on{Br}(F)$ is denoted additively. If $\on{char}(F)\neq 2$ and $a,b\in F^\times$, we write $(a,b)$ for the corresponding quaternion algebra over $F$ and for its class in $\on{Br}(F)$. We write $N_{a_1,\dots,a_n}\colon\on{Br}(F_{a_1,\dots,a_n})\to \on{Br}(F)$ for the corestriction map along $F_{a_1,\dots,a_n}/F$.

	An $F$-variety is a separated integral $F$-scheme of finite type. If $X$ is an $F$-variety, we denote by $F(X)$ the function field of $X$. If $x$ is a point of $X$, we denote by $O_{X,x}$ the local ring of $X$ at $x$ and by $F(x)$ the residue field of $x$.
	
	\section{Preliminaries}\label{preliminaries}

 \subsection{Galois algebras}\label{kummer-sub} 

Let $G$ be a finite (abstract) group. By definition, a \emph{$G$-algebra} $L/F$ is an \'etale $F$-algebra on which $G$ acts via $F$-algebra automorphisms. We say that the $G$-algebra $L$ is \emph{Galois} if $|G|=\dim_FL$ and $L^G=F$; see \cite[Definitions (18.15)]{knus1998book}. A $G$-algebra $L/F$ is Galois if and only if the morphism of schemes $\on{Spec}(L)\to \on{Spec}(F)$ is an \'etale $G$-torsor. By  \cite[Example (28.15)]{knus1998book}, we have a canonical bijection 
\begin{equation}\label{galois-alg}
H^1(F,G)\xrightarrow{\sim}\set{\text{Isomorphism classes of Galois $G$-algebras over $F$}}
\end{equation}
which is functorial in $F$ and $G$.

Suppose now that $G=N\rtimes S$, where $N$ is a normal subgroup of $G$ and $S$ is a subgroup of $G$. Let $E$ be a Galois $S$-algebra over $F$, and let $\pi\colon\Gamma_F\to S$ be a continuous group homomorphism whose class in $H^1(F,S)$ coincides with the class of $E/F$. The group $S$ acts on $N$ by conjugation. We view $N$ as a $\Gamma_F$-module via $\pi$, and we write $N_{\pi}$ for the twist; see \cite[28.C]{knus1998book}. We have a canonical bijection
\begin{align}\label{twist-galois-algebras}
    H^1(F,N_{\pi})\xrightarrow{\sim}\{\text{Isom. classes of pairs $(K,\varphi)$, where $K/F$ is a Galois $G$-algebra}\\ \text{and $\varphi \colon K^N\to E$ is an isomorphism of Galois $S$-algebras}\},\nonumber
\end{align}
which is functorial in $F$.
Here, an isomorphism of pairs $(K,\varphi)\to(K',\varphi')$ is defined as an isomorphism of Galois $G$-algebras $\sigma\colon K\to K'$ over $F$ such that $\varphi=\varphi'\circ \sigma$ on $K^N$. Under the identification of (\ref{twist-galois-algebras}), the surjective twisting map 
\begin{equation}\label{twisting-map}
    H^1(F,N_{\pi})\to \on{Fiber}_{\pi}[H^1(F,G)\to H^1(F,S)]
\end{equation}
of \cite[Proposition 28.11]{knus1998book} sends the class of a pair $(K,\varphi)$ to the class of $K$.

\subsection{Kummer Theory} Suppose that $\on{char}(F)\neq 2$. Then $\Z/2\Z \simeq \mu_2$ over $F$, and so the Kummer sequence yields an isomorphism
	\begin{equation}\label{kummer}\on{Hom}(\Gamma_F,\Z/2\Z )=H^1(F,\Z/2\Z )\simeq H^1(F,\mu_2)\simeq F^{\times}/F^{\times 2}.\end{equation}

For every $a\in F^{\times}$, we let $\chi_a\colon \Gamma_F\to \Z/2\Z $ be the homomorphism corresponding to the coset $a F^{\times 2}$ under (\ref{kummer}). Explicitly, letting $a'\in F_{\on{sep}}^\times$ be such that $(a')^2=a$, we have $g(a')=(-1)^{\chi_a(g)}a'$ for all $g\in \Gamma_F$. The definition of $\chi_a$ does not depend on the choice of $a'$.

Now let $n\geq 1$. For all $i=1,\dots,n$, let $\sigma_i$ be the canonical generator of the $i$-th factor $\Z/2\Z$ of $(\Z/2\Z)^n$. It follows from (\ref{kummer}) that all Galois $(\Z/2\Z)^n$-algebras over $F$ are of the form $F_{a_1,\dots,a_n}$, where $a_1,\dots,a_n\in F^\times$ and the Galois $(\Z/2\Z)^n$-algebra structure is defined by $\sigma_i(\sqrt{a_i})=-\sqrt{a_i}$ for all $i$ and $\sigma_i(\sqrt{a_j})=\sqrt{a_j}$ for all $j\neq i$. The elements $a_1,\dots,a_n$ are uniquely determined modulo $F^{\times 2}$.
	
	The Kummer sequence provides a group isomorphism \[\iota\colon H^2(F,\Z/2\Z)\xrightarrow{\sim}\on{Br}(F)[2].\] For all $a,b\in F^{\times}$, $\iota(\chi_a\cup \chi_b)$ is equal to $(a,b)$; see \cite[Chapter XIV, Proposition 5]{serre1979local}. 
 The next lemma is well known and will be used several times in what follows.
	
	\begin{lemma}\label{cup-norm}
		Suppose that $\on{char}(F)\neq 2$ and let $a,b\in F^{\times}$. The following are equivalent:
		
		(i) $(a,b)=0$ in $\on{Br}(F)$;
		
		(ii) $b\in N_a$;
		
		(iii) $a\in N_b$;
  
		(iv) the smooth projective $F$-conic of equation $aX^2+bY^2=Z^2$ has an $F$-point (equivalently, it is isomorphic to $\P^1_F$).
	\end{lemma}
	
	\begin{proof}
  See \cite[Propositions 1.1.7 and 1.3.2]{gille2017central}.
	\end{proof}

	\subsection{Massey products}\label{massey-sub}
	Let $(A,\partial)$ be a DGA over $\F_2$, that is, a graded associative $\F_2$-algebra $A=\oplus_{i\geq 0}A^i$ with a homomorphism $\partial\colon A\to A$ of degree $1$ (called the differential) such that $\partial\circ \partial=0$ and $\partial(ab)=\partial(a)b+a\partial(b)$ for all $a,b\in A$. We denote by $H^{\smallbullet}(A)\coloneqq \on{Ker}\partial/\on{Im}\partial$ the cohomology of $(A,\partial)$, and we write $\cup$ for the multiplication (cup product) on $H^{\smallbullet}(A)$.
	
	Let $n\geq 2$ be an integer and $a_1,\dots,a_n\in H^1(A)$. Consider a collection $M=(a_{ij})$ of elements of $A^1$, where $1\leq i<j\leq n+1$, $(i,j)\neq (1,n+1)$. We say that $M$ is a \emph{defining system} for the $n$-th order Massey product $\ang{a_1,\dots,a_n}$ if 

 \begin{enumerate}
     \item $\partial(a_{i,i+1})=0$ and $a_{i,i+1}$ represents $a_i$ in $H^1(A)$.
     \item $\partial(a_{ij})=\sum_{l=i+1}^{j-1}a_{il}a_{lj}$ for all $i<j-1$.
 \end{enumerate}
	
	It follows from (2) that  $\sum_{l=2}^na_{1l}a_{l,n+1}$ is a $2$-cocycle. We write $\ang{a_1,\dots,a_n}_M$ for its cohomology class in $H^2(A)$. By definition, the \emph{Massey product} of $a_1,\dots,a_n$ is the subset $\ang{a_1,\dots,a_n}\subset H^2(A)$ of elements of the form $\ang{a_1,\dots,a_n}_M$ for some defining system $M$.

	We say that the Massey product $\ang{a_1,\dots,a_n}$ is \emph{defined} if it is non-empty, that is, if there exists a defining system for $\ang{a_1,\dots,a_n}$. We say that $\ang{a_1,\dots,a_n}$ \emph{vanishes} if $0\in \ang{a_1,\dots,a_n}$. If $\ang{a_1,\dots,a_n}$ vanishes, then it is defined.
	
	\begin{rmk}\label{cup-products-zero}
	Suppose that $n\geq 3$, and let $a_1,\dots,a_n\in H^1(A)$ and $M=(a_{ij})$ be a defining system for the Massey product $\ang{a_1,\dots,a_n}$. As a special case of (2), $\partial(a_{i,i+2})=a_{i,i+1} a_{i+1,i+2}$ represents $a_i\cup a_{i+1}$ in $H^2(A)$. Thus, if $\ang{a_1,\dots,a_n}$ is defined, then $a_1\cup a_2=a_2\cup a_3=\dots= a_{n-1}\cup a_n=0$ in $H^2(A)$.
	\end{rmk}
	
	\begin{example}\label{massey-galois-ex}
		(a) Let $\Gamma$ be a profinite group. The complex $(C^{\smallbullet}(\Gamma,\Z/2\Z),\partial)$ of mod $2$ non-homogeneous continuous cochains  of $\Gamma$ with the standard cup product is a DGA over $\F_2$. We write $H^{\smallbullet}(\Gamma,\Z/2\Z)$ for the cohomology of $(C^{\smallbullet}(\Gamma,\Z/2\Z),\partial)$. By the previous discussion, it makes sense to talk about (mod $2$) Massey products of elements of $H^1(\Gamma,\Z/2\Z)$.
		
		\smallskip
		
		(b) As a special case of (a) when $\Gamma=\Gamma_F$, we may talk about (mod $2$) Massey products of elements of $H^1(F,\Z/2\Z)$. Suppose that $\on{char}(F)\neq 2$. By \Cref{kummer-sub} elements of $H^1(F,\Z/2\Z)$ correspond to elements of $F^\times/F^{\times 2}$, hence it makes sense to talk about Massey products of elements of $F^\times$. If $a_1,\dots,a_n\in F^\times$, we denote by $\ang{a_1,\dots,a_n}\subset H^2(F,\Z/2\Z)$ their (mod $2$) Massey product. By definition, $\ang{a_1,\dots,a_n}$ is defined (resp. vanishes) if and only if  so does $\ang{\chi_{a_1},\dots,\chi_{a_n}}$.
	\end{example} 
	  
 Let $U_{n+1}\subset \on{GL}_{n+1}(\F_2)$ be the subgroup of all upper-triangular unipotent matrices. For all $1\leq i<j\leq n+1$, we denote by $u_{i,j}\colon U_{n+1}\to \Z/2\Z$ the $(i,j)$-th coordinate function on $U_{n+1}$. The functions $u_{i,j}$ are group homomorphisms only when $j=i+1$. The center $Z_{n+1}$ of $U_{n+1}$ is the subgroup of all matrices such that $u_{i,j}=0$ when $(i,j)\neq (1,n+1)$; in particular $Z_{n+1}\simeq \Z/2\Z$. We have a commutative diagram
	\begin{equation}\label{central-ses}
	\begin{tikzcd}
		1\arrow[r] & Z_{n+1}\arrow[r] &  U_{n+1}\arrow[r]  \arrow[dr]& \cl{U}_{n+1}\arrow[r] \arrow[d]  & 1\\
		&&& (\Z/2\Z)^n
	\end{tikzcd}	
	\end{equation}
	where the row is exact and the homomorphism $U_{n+1}\to (\Z/2\Z)^n$ is given by $(u_{1,2},u_{2,3},\dots, u_{n,n+1})$. The group $\cl{U}_{n+1}$ may be identified with the group of all upper triangular unipotent matrices of size $(n+1)\times (n+1)$ with the entry $(1,n+1)$ omitted.  
	
 The following result is due to Dwyer \cite{dwyer1975homology}.
 
	\begin{thm}\label{dwyer}
  Let $\Gamma$ be a profinite group, $\chi= (\chi_1,\dots,\chi_n)\colon \Gamma\to (\Z/2\Z)^n$ be a continuous group homomorphism, and consider the mod $2$ Massey product $\ang{\chi_1,\dots,\chi_n}$.
  
  (i) If $\tilde{\chi} : \Gamma\to \cl{U}_{n+1}$ is a continuous homomorphism lifting $\chi$, then the collection of 1-cochains $M_{\tilde{\chi}}\coloneqq \set{u_{i,j}\circ \tilde{\chi}}$ forms a defining system for $\ang{\chi_1,\dots,\chi_n}$.
    
(ii) There is a bijective correspondence between continuous lifts  of $\chi$ and defining systems for $\ang{\chi_1,\dots,\chi_n}$, given by sending a lift $\tilde{\chi} : \Gamma\to \cl{U}_{n+1}$ to $M_{\tilde{\chi}}$.
     
(iii) Let $\partial\colon H^1(\Gamma,\cl{U}_{n+1})\to H^2(\Gamma,\Z/2\Z)$ be the connecting homomorphism induced by (\ref{central-ses}). Then \[\ang{\chi_1,\dots,\chi_n}=\set{\partial([\tilde{\chi}]): \tilde{\chi}\colon \Gamma\to \cl{U}_{n+1} \text{ is a continuous lift of $\chi$}}.\]
  
(iv) In particular, $\ang{\chi_1,\dots,\chi_n}$ is defined (resp. vanishes) if and only if $\chi$ lifts to a continuous homomorphism $\Gamma\to \cl{U}_{n+1}$ (resp. $\Gamma\to U_{n+1}$). 
	\end{thm}
	
	\begin{proof}
	See \cite{dwyer1975homology} or \cite[Proposition 2.2]{harpaz2019massey}. Although the theorem had originally been stated by Dwyer for discrete groups $\Gamma$, it can be readily extended to the case of profinite groups.
	\end{proof}

\begin{cor}\label{dwyer-pro-free}
    Let $\Gamma$ be a profinite group such that the group $H^2(\Gamma,\Z/2)$ is trivial. Then, for all $n\geq 1$ and all continuous homomorphisms $\chi_1,\dots,\chi_n\colon\Gamma\to \Z/2\Z$, the Massey product $\ang{\chi_1,\dots,\chi_n}$ is either empty or equal to $\set{0}$.
\end{cor}
    The assumptions of \Cref{dwyer-pro-free} are satisfied, for example, when $\Gamma$ is the absolute Galois group of a field of characteristic $2$ or of a finite field; see \cite[II.\S 3.3(a) and II.\S 2 Proposition 3]{serre2002galois}.
\begin{proof}
   This is an immediate consequence of \Cref{dwyer}(iii).
\end{proof}

We also let 
    \[Q_{n+1}\coloneqq \on{Ker}(U_{n+1}\to (\Z/2\Z)^n),\quad \cl{Q}_{n+1}\coloneqq \on{Ker}(\cl{U}_{n+1}\to (\Z/2\Z)^n)=Q_{n+1}/Z_{n+1}.\]
The induced maps
 \[H^1(F,U_{n+1})\to H^1(F,(\Z/2\Z)^n),\qquad H^1(F,\cl{U}_{n+1})\to H^1(F,(\Z/2\Z)^n)\]
 send a Galois $U_{n+1}$-algebra $L/F$ (resp. a Galois $\cl{U}_{n+1}$-algebra $K/F$) to the $(\Z/2\Z)^n$-algebra $L^{Q_{n+1}}/F$ (resp. $K^{\cl{Q}_{n+1}}/F$).

 \begin{cor}\label{dwyer-cor}
Suppose that $\on{char}(F)\neq 2$, and let $a_1,\dots,a_n\in F^\times$. Then the Massey product $\ang{a_1,\dots,a_n}$ is defined (resp. vanishes) if and only if there exists a Galois $\cl{U}_{n+1}$-algebra $K/F$ (resp. Galois $U_{n+1}$-algebra $L/F$)  such that $K^{\cl{Q}_{n+1}}=F_{a_1,\dots,a_n}$ (resp. $L^{{Q}_{n+1}}=F_{a_1,\dots,a_n}$).
 \end{cor}

\begin{proof}
    Immediate consequence of \Cref{dwyer}(iv), \Cref{massey-galois-ex}(b) and (\ref{galois-alg}).
\end{proof}

		We conclude \Cref{massey-sub} with some observations. We have a cartesian square of groups
	\begin{equation}\label{phi-phi'}
	\begin{tikzcd}
		\cl{U}_{n+1} \arrow[d,"\varphi'_{n+1}"] \arrow[r,"\varphi_{n+1}"] & U_n \arrow[d,"\varphi'_n"] \\
		U_n \arrow[r,"\varphi_n"] & U_{n-1}	
	\end{tikzcd}
	\end{equation}
	where $\phi_n$ and $\phi_{n+1}$ (resp.  $\phi'_n$ and $\phi'_{n+1}$) are given by forgetting the right-most column (resp. the top row).  
	\begin{prop}\label{glueing}
		The map
    \[H^1(F,\cl{U}_{n+1})\to H^1(F,U_n)\times_{(\phi_n)_*,H^1(F,U_{n-1}),(\phi'_n)_*} H^1(F,U_n)\]
    induced by $(\phi'_{n+1}, \phi'_n)$ is surjective.
	\end{prop}
	
	\begin{proof}
		Since (\ref{phi-phi'}) is cartesian, to give a continuous group homomorphism $\Gamma_F\to \cl{U}_{n+1}$ is the same as giving two continuous homomorphisms $f,f'\colon \Gamma_F\to U_n$ such that $\varphi_n\circ f=\varphi_n'\circ f'$. It follows that the square of pointed sets
		\begin{equation}\label{hom-cartesian}
		\begin{tikzcd}
			\on{Hom}(\Gamma_F,\cl{U}_{n+1}) \arrow[d,"(\varphi'_{n+1})_*"] \arrow[r,"(\varphi_{n+1})_*"] & \on{Hom}(\Gamma_F,U_n) \arrow[d,"(\varphi'_n)_*"] \\
			\on{Hom}(\Gamma_F,U_n) \arrow[r,"(\varphi_n)_*"] & \on{Hom}(\Gamma_F,U_{n-1})	
		\end{tikzcd}
		\end{equation}
		is also cartesian. 
		
		Recall that, for every finite group $G$, we have $H^1(F,G)=\on{Hom}(\Gamma_F,G)/\sim$, where if $\psi,\psi'\colon\Gamma_F\to G$ are continuous group homomorphisms, we write $\psi\sim \psi'$ if and only if there exists $g\in G$ such that $g\psi' g^{-1}=\psi$.  Suppose given $\psi,\psi'\colon\Gamma_F\to U_n$ and $g\in U_{n-1}$ such that \[g((\phi'_n)_*(\psi'))g^{-1}=(\phi_n)_*(\psi)\quad \text{in $\on{Hom}(\Gamma_F,U_{n-1})$.}\] 
  Let $\tilde{g}\in U_n$ be such that $\phi'_n(\tilde{g})=g$. Then $(\phi'_n)_*(\tilde{g}\psi'\tilde{g}^{-1})=(\phi_n)_*(\psi)$. Since (\ref{hom-cartesian}) is cartesian, there exists $\Psi\in \on{Hom}(\Gamma_F,\cl{U}_{n+1})$ lifting $\psi$ and $\tilde{g}\psi'\tilde{g}^{-1}$. This completes the proof.
	\end{proof}
	
\begin{lemma}\label{differ-by-scalar}
		Suppose that $\on{char}(F)\neq 2$. Let $n\geq 2$ be an integer and $K$ be a Galois $\cl{U}_{n+1}$-algebra over $F$. Suppose that there exists a Galois $U_{n+1}$-algebra $L$ such that $L^{Z_{n+1}}\simeq K$, and write $L=K_\pi$ for some $\pi\in K^\times$. Then:
  
  (a) for all $t\in F^\times$ the $F$-algebra $K_{t\pi}$ has the structure of a Galois $U_{n+1}$-algebra such that $K_{t\pi}^{Z_{n+1}}=K$, and
  
  (b) all Galois $U_{n+1}$-algebras $E$ such that $E^{Z_{n+1}}\simeq K$ arise in this way.
	\end{lemma}
	
	\begin{proof}
  Passing to Galois cohomology in (\ref{central-ses}) yields an exact sequence of pointed sets
		\[1\to F^\times/F^{\times 2}\to H^1(F,U_{n+1})\to H^1(F,\cl{U}_{n+1}).\]
		The group $F^\times/F^{\times 2}$ acts transitively on the fibers of $H^1(F,U_{n+1})\to H^1(F,\cl{U}_{n+1})$. A simple cocycle calculation shows that $t\in F^\times/F^{\times 2}$ sends the class of $K_\pi/F$  to the class of $K_{t\pi}/F$. This proves (a) and (b) at once. 
	\end{proof}

	\subsection{Specialization}\label{specialize-sub}
    In this section, we prove some specialization properties of Massey products. They can be useful to avoid case-by-case analysis, for example when certain quantities become zero or undefined.
    
	\begin{lemma}\label{specialize}
		Let $R$ be a discrete valuation ring with fraction field $K$ and residue field $k$, and suppose that $\on{char}(k)\neq 2$. Let $a_1,\dots,a_n\in R^\times$, and let $\cl{a}_1,\dots,\cl{a}_n\in k^{\times}$ be their reductions. 
		
		(a) If the Massey product $\ang{a_1,\dots,a_n}$ is defined over $K$, then $\ang{\cl{a}_1,\dots,\cl{a}_n}$ is defined over $k$.
		
		(b) If the Massey product $\ang{a_1,\dots,a_n}$ vanishes over $K$, then $\ang{\cl{a}_1,\dots,\cl{a}_n}$ vanishes over $k$.
	\end{lemma}

	\begin{proof}
		The completion of $R$ has residue field equal to $k$ and fraction field containing $K$. We may thus replace $R$ by its completion and assume that $R$ is complete.  
		
		Let $q\geq 1$ be the characteristic exponent of $K$, that is, $q=\on{char}(K)$ if $\on{char}(K)>0$ and $q=1$ if $\on{char}(K)=0$. Let $\pi\in R$ be a uniformizer. For every $d\geq 1$ prime to $q$, choose a $d$-th root $\pi^{1/d}$ of $\pi$ such that for all $d_1,d_2\geq 1$ prime to $q$ we have $(\pi^{1/d_1d_2})^{d_1}=\pi^{1/d_2}$. Define $K_\infty\coloneqq \cup_d K(\pi^{1/d})$, where $d\geq 1$ ranges over all integers prime to $q$.
		Let $\Delta \coloneqq  \Gamma_{K_{\infty}}$, and define $L \coloneqq  K_{\on{nr}} K_{\infty}$. Then $\Gamma_L$ is a pro-$q$-group (trivial if $q=1$) and 
		$\on{Gal}(L/K_{\infty}) = \on{Gal}(K_{\on{nr}}/K) = \Gamma_k$.
		It follows that for every finite $2$-group $G$ the natural map
		\[\on{Hom}(\Gamma_k, G) = \on{Hom}(\on{Gal}(L/K_{\infty}), G) \to \on{Hom}(\Delta, G)\]
		is an isomorphism. The map
		\[\on{Hom}(\Gamma_K, G) \to \on{Hom}(\Delta, G) = \on{Hom}(\Gamma_k, G),\] 
		induces a specialization map
		\[s\colon H^1(K,G)\to H^1(k,G).\]
		which is covariant in $G$. 
		
        For all $a\in R^\times$, we denote by $\cl{a}\in k^\times$ the reduction of $a$ modulo the maximal ideal of $R$. Under the identification of (\ref{kummer}), $s\colon H^1(K,\Z/2\Z)\to H^1(k,\Z/2\Z)$ sends $a K^{\times 2}$ to $\cl{a} k^{\times 2}$. Therefore for all $n\geq 1$:
		\begin{equation}\label{true-specialization}\text{The map $s\colon H^1(K,(\Z/2\Z)^n)\to H^1(k,(\Z/2\Z)^n)$ sends $(a_1,\dots,a_n)$ to $(\cl{a}_1,\dots,\cl{a}_n)$.}
		\end{equation}
		
		(a) We have a commutative square of pointed sets
		\begin{equation}\label{special}
			\begin{tikzcd}
				H^1(K,\cl{U}_{n+1})  \arrow[d] \arrow[r,"s"] &  H^1(k,\cl{U}_{n+1}) \arrow[d]  \\  
				H^1(K, (\Z/2\Z)^n) \arrow[r,"s"]  &  H^1(k, (\Z/2\Z)^n).
			\end{tikzcd}
		\end{equation}
		By \Cref{dwyer-cor}, $(a_1,\dots,a_n)$ lifts to $H^1(K,\cl{U}_{n+1})$, hence (\ref{true-specialization}) and (\ref{special}) imply that  $(\cl{a}_1,\dots,\cl{a}_n)\in H^1(k, (\Z/2\Z)^n)$ lifts to $H^1(k,\cl{U}_{n+1})$. The conclusion follows from \Cref{dwyer}(iv).
		
		(b) We may argue as in the proof of (a), replacing $\cl{U}_{n+1}$ by $U_{n+1}$ everywhere.
	\end{proof}
	\begin{prop}\label{specialize-variety}
		Let $a_1,\dots,a_n\in F^{\times}$, let $X$ be an $F$-variety with a regular $F$-point. Then we have the following.
		
		(a) The Massey product $\ang{a_1,\dots,a_n}$ is defined over $F$ if and only if it is defined over $F(X)$.
		
		(b) The Massey product $\ang{a_1,\dots,a_n}$ vanishes over $F$ if and only if it vanishes over $F(X)$.
	\end{prop}
	
	\begin{proof}
		It is clear that if $\ang{a_1,\dots,a_n}$ is defined (resp. it vanishes) over $F$, then it is defined (resp. it vanishes) over $F(X)$.
		
		For the converse, let $x\in X(F)$ be a regular $F$-point, that is, the local ring $O_{X,x}$ is regular.  Let $d$ be the dimension of $X$ and $t_1,\dots,t_d\in O_{X,x}$ be a regular system of parameters. For every $i=1,\dots,d$, let $O_i$ be the localization of $O_{X,x}/(t_1,\dots,t_{i-1})$ at the prime ideal generated by the image of $t_i$. Since $O_{X,x}$ is regular, the quotient of a regular local ring by a non-zero divisor is regular, and the localization of a regular local ring is regular, every $O_i$ is a regular local ring of dimension $1$, that is, a discrete valuation ring. Moreover, the fraction field of $O_1$ is $F(X)$, the residue field of $O_d$ is $F$, and for all $i=1,\dots,d-1$ the residue field of $O_i$ coincides with the fraction field of $O_{i+1}$. Now (a) (resp. (b)) follows from \Cref{specialize}(a) (resp. (b)), applied to $R=O_i$ for $i=1,\dots,d$.
	\end{proof}

	\begin{rmk}
	Let $p$ be a prime number, and suppose that $\on{char}(F)\neq p$ and that $F$ contains a primitive $p$-th root of unity. Then the definition of mod $p$ Massey product given in \Cref{massey-galois-ex}(b) for $p=2$ extends to all $p$. The constructions and the results of \Cref{massey-sub} and \Cref{specialize-sub} 
 extend to arbitrary $p$, with the same proofs.
	\end{rmk}

	\section{Massey products and Galois algebras}\label{u5-bar-section}

    From now on in this paper, we suppose that $\on{char}(F)\neq 2$. 
	
	\subsection{Galois \texorpdfstring{$U_3$}{U3}-algebras}
	
	Let $a,b\in F^\times$, and suppose that $(a,b)=0$ in $\on{Br}(F)$. We write $(\Z/2\Z)^2=\ang{\sigma_a,\sigma_b}$, and we view $F_{a,b}$ as a Galois $(\Z/2\Z)^2$-algebra as in \Cref{kummer-sub}.
	Let $\alpha\in F_a^\times$ satisfy $N_a(\alpha)=bx^2$ for some $x\in F^\times$, and consider the  \'etale $F$-algebra $(F_{a,b})_\alpha$. We have
	\begin{align*}U_3=\ang{\sigma_a,\sigma_b: \sigma_a^2=\sigma_b^2=[\sigma_a,\sigma_b]^2=1}.	
	\end{align*}	
	Moreover, $\cl{U}_3=(\Z/2\Z)^2$ and the surjective homomorphism $U_3\to \cl{U}_3$ is given by $\sigma_a\mapsto \sigma_a$ and $\sigma_b\mapsto \sigma_b$.   
 Observe that $\sigma_a(\alpha)=bx^2/\alpha$ and $\sigma_b(\alpha)=\alpha$. We may thus define a Galois $U_3$-algebra structure on $(F_{a,b})_\alpha$ by letting $U_3$ act on $F_{a,b}$ via $\cl{U}_3$ and by setting	
	\begin{equation}\label{u3-algebra}
		\sigma_a(\sqrt{\alpha})=x\sqrt{b}/\sqrt{\alpha},\qquad \sigma_b(\sqrt{\alpha})=\sqrt{\alpha}.
	\end{equation} 
    
We leave to the reader the verification that $\sigma_a^2=\sigma_b^2=[\sigma_a,\sigma_b]^2=1$ on $(F_{a,b})_\alpha$, that $(F_{a,b})_{\alpha}$ is a Galois $U_3$-algebra and that the subalgebra of $Q_3$-invariants is $F_{a,b}$.
 
	\begin{lemma}\label{minus-x}
	Let $K^+\coloneqq (F_{a,b})_\alpha$ with the Galois $U_3$-algebra structure of (\ref{u3-algebra}) and $K^-\coloneqq(F_{a,b})_\alpha$, with the Galois $U_3$-algebra structure given by 
 \begin{equation}\label{u3-algebra-minus}
		\sigma_a(\sqrt{\alpha})=-x\sqrt{b}/\sqrt{\alpha},\qquad \sigma_b(\sqrt{\alpha})=\sqrt{\alpha}.
	\end{equation} Then $\sigma_b\colon K^+\to K^{-}$ is an isomorphism of Galois $U_3$-algebras.
	\end{lemma}
	
    Note that $\sigma_b\colon K^+\to K^{-}$ is not $F_{a,b}$-linear.
 
	\begin{proof}
	For all $\eta \in (F_{a,b})_{\alpha}$, we write $\eta^+$ and $\eta^-$ for the element $\eta$, viewed as an element of $K^+$ and $K^-$, respectively. Then $\sigma_b\colon K^+\to K^{-}$ sends $\eta^+$ to $\sigma_b(\eta^+)=\sigma_b(\eta)^-$. Since $\sigma_b$ is an isomorphism of \'etale algebras, it suffices to check that it is compatible with the Galois $U_3$-algebra structures. This follows from the fact that $\sigma_a\sigma_b=\sigma_b\sigma_a$ on $F_{a,b}$ and
 \[\sigma_b(\sigma_a(\sqrt{\alpha}^+))=\sigma_b(x\sqrt{b}^+/\sqrt{\alpha}^+)=-x\sqrt{b}^-/\sqrt{\alpha}^-=\sigma_a(\sqrt{\alpha}^-)=\sigma_a(\sigma_b(\sqrt{\alpha}^+)).\qedhere\]
	\end{proof}
	
	Symmetrically, if 
	$\beta\in F_b^\times$ satisfies $N_b(\beta)=ay^2$ for some $y\in F^\times$, 
	the \'etale $F$-algebra $(F_{a,b})_\beta$ has structure of a Galois $U_3$-algebra defined by
	\begin{equation}\label{u3-algebra-2}\sigma_a(\sqrt{\beta})=\sqrt{\beta},\qquad \sigma_b(\sqrt{\beta})=y\sqrt{a}/\sqrt{\beta}.\end{equation}
	
	\begin{prop}\label{uu3}
		Let $a,b\in F^\times$.
  
		(a) Every Galois $U_3$-algebra $K$ over $F$ such that $K^{Q_3}=F_{a,b}$ is of the form $(F_{a,b})_\alpha$ for some $\alpha\in F_a^\times$ with the property
		$N_a(\alpha)=b$ in $F^\times/F^{\times 2}$ and $U_3$-algebra structure as in (\ref{u3-algebra}).
  
        (b) Every Galois $U_3$-algebra $K$ over $F$ such that $K^{Q_3}=F_{a,b}$ is of the form $(F_{a,b})_\beta$ for some $\beta\in F_b^\times$ with the property
		$N_b(\beta)=a$ in $F^\times/F^{\times 2}$ and $U_3$-algebra structure as in (\ref{u3-algebra-2}).
		
		(c) Let $\alpha\in F_a^\times$, $\beta \in F_b^\times$ be such that $N_a(\alpha)=b$ in $F^\times/F^{\times 2}$ and $N_b(\beta)=a$ in $F^\times/F^{\times 2}$. The two Galois $U_3$-algebras $(F_{a,b})_\alpha$ and $(F_{a,b})_\beta$ (with $U_3$-algebra structure as in (\ref{u3-algebra}) and (\ref{u3-algebra-2}), respectively) are isomorphic if and only there exists $\omega\in F_{a,b}^\times$ such that
    \[\alpha\beta=\omega^2,\qquad (\sigma_a-1)(\sigma_b-1)\omega=-1.\]
	\end{prop}
 
	\begin{proof}       
    	(a) We have $U_3=N\rtimes S$, where
     \[
        N=\begin{bmatrix}
			1 & 0 & *  \\
			& 1 & *  \\
			&    & 1
		\end{bmatrix},\qquad 
		 S=\begin{bmatrix}
			1 & * & 0  \\
			& 1 & 0  \\
			&    & 1
		\end{bmatrix}.
  \]
     We let $S$ act on $N$ by conjugation. Then $N$ has a permutation basis given by
        \[\begin{bmatrix}
			1 & 0 & 0  \\
			& 1 & 1  \\
			&    & 1
		\end{bmatrix},\qquad \begin{bmatrix}
			1 & 0 & 1  \\
			& 1 & 1  \\
			&    & 1
		\end{bmatrix}.\]
        Therefore we have a commutative square of $S$-modules:
        \begin{equation}\label{ind-group-s}
    \begin{tikzcd}
    N \arrow[r,"\sim"] \arrow[d,"u_{23}"] & \on{Ind}_{\set{1}}^{S}(\Z/2\Z) \arrow[d,"\on{Norm}"] \\
    \Z/2\Z \arrow[r,equal] & \Z/2\Z.
    \end{tikzcd}
  \end{equation}

        Let $p\colon U_3\to S$ be the projection map, and let $N_p$ be the twist of $N$ by $p$; see \cite[28.C]{knus1998book}. Then (\ref{ind-group-s}) induces a commutative square of $U_3$-modules:
    \begin{equation}\label{ind-group}
    \begin{tikzcd}
    N_p \arrow[r,"\sim"] \arrow[d,"(u_{23})_p"] & \on{Ind}_{N}^{U_3}(\Z/2\Z) \arrow[d,"\on{Norm}"] \\
    \Z/2\Z \arrow[r,equal] & \Z/2\Z.
    \end{tikzcd}
  \end{equation}

Let $\rho\colon\Gamma_F\to U_3$ be a continuous homomorphism whose class in $H^1(F,U_3)$ coincides with that of $K$, and define $\pi\colon \Gamma_F\to S$ by $\pi\coloneqq p\circ \rho$. Since $K^N=F_a$, the class of $\pi$ in $H^1(F,S)=H^1(F,\Z/2\Z)$ is equal to $\chi_a$.  Pullback of (\ref{ind-group}) along $\rho$ yields the following commutative square:
  \begin{equation}\label{ind-profinite}
    \begin{tikzcd}
    N_{\pi} \arrow[r,"\sim"] \arrow[d,"(u_{23})_{\pi}"] & \on{Ind}_{F_a}^F(\Z/2\Z) \arrow[d,"N_a"] \\
    \Z/2\Z \arrow[r,equal] & \Z/2\Z. 
    \end{tikzcd}
  \end{equation}
  Here $\on{Ind}_{F_a}^F(\Z/2\Z)$ indicates the $\Gamma_F$-module corresponding to the pushforward of the constant \'etale sheaf $\Z/2\Z$ on $\Spec(F_a)$ to $\Spec(F)$; see \cite[Theorem II.1.9]{milne1980etale}. Combining (\ref{ind-profinite}) with Faddeev--Shapiro's lemma, we deduce that the composition
        \[
            \Phi\colon H^1(F,N_{\pi})\xrightarrow{\on{Res}} H^1(F_a,N)\xrightarrow{u_{13}} H^1(F_a,\Z/2\Z)
        \]
    is an isomorphism fitting in the commutative diagram
\begin{equation}\label{shapiro2}
\begin{tikzcd}
    H^1(F,N_{\pi}) \arrow[r,"\Phi"] \arrow[d,"(u_{23})_{\pi}"] & H^1(F_a,\Z/2\Z) \arrow[r,equal]\arrow[d,"N_a"]  & F_a^\times/F_a^{\times 2} \arrow[d,"N_a"] \\
    H^1(F,\Z/2\Z) \arrow[r,equal] & H^1(F,\Z/2\Z) \arrow[r,equal] & F^{\times}/F^{\times 2}.  
\end{tikzcd}
\end{equation}
    Let $E$ be a Galois $U_3$-algebra and $\varphi\colon E^N\to F_a$ be an isomorphism of Galois $U_3$-algebras over $F$. By base change, we obtain an isomorphism of Galois $U_3$-algebras $\varphi_{F_a}\colon (E^N)_a\to (F_a)_a=\prod_{\sigma\in S}F_a$ over $F_a$. Therefore, we may write $E_a=\prod_{\sigma\in S}E_{\varphi,\sigma}$, where $E_{\varphi,\sigma}$ is the subalgebra of $E_a$ lying over the inverse image of the factor in $\prod_{\sigma\in S}F_a$ corresponding to $\sigma$. In terms of the identification (\ref{twist-galois-algebras}), the Faddeev--Shapiro isomorphism $\Phi$ sends the class of the pair $(E,\varphi)$ to the class of the $\Z/2\Z$-Galois algebra $E_{\varphi,0}/F_a$, where $0\in S$ is the identity element.

    Since $K^{Q_3}=F_{a,b}$, the pair $(K,\on{id})$ defines an element in $H^1(F,N_{\pi})$. Let $\alpha\in F_a^\times$ be such that $\Phi$ sends the class of $(K,\on{id})$ to $(F_a)_{\alpha}/F_a$. By (\ref{shapiro2}), we have $N_a(\alpha)=bx^2$ for some $x\in F^\times$. The pair $((F_{a,b})_{\alpha},\on{id})$, where $(F_{a,b})_{\alpha}$ is the Galois $U_3$-algebra of (\ref{u3-algebra}), also defines an element of $H^1(F,N_{\pi})$ which is mapped to $(F_a)_{\alpha}/F_a$ by $\Phi$. The injectivity of $\Phi$ now implies that the Galois $U_3$-algebras $K$ and $(F_{a,b})_{\alpha}$ are isomorphic over $F$, as desired.

        (b) Analogous to (a), replacing $N$ and $S$ by
        \[
		N'=\begin{bmatrix}
			1 & * & *  \\
			& 1 & 0  \\
			&    & 1
		\end{bmatrix}\quad \text{and}
		\quad S'=\begin{bmatrix}
			1 & 0 & 0  \\
			& 1 & *  \\
			&    & 1
		\end{bmatrix},
		\]
        respectively.
  
		(c) Suppose given an isomorphism between two $U_3$-algebras $(F_{a,b})_\alpha$ and $(F_{a,b})_\beta$. 
		Write $\beta'\in (F_{a,b})_\alpha$ for the image of $\beta$ under the isomorphism and set 
		$\omega\coloneqq\sqrt{\alpha}\cdot\sqrt{\beta'}$. We have $[\sigma_a,\sigma_b]\omega=\omega$, that is, $\omega$ is invariant under $Z_3=Q_3$, and hence
		$\omega\in F_{a,b}^\times$. Moreover,
		\[
		(\sigma_a-1)(\sigma_b-1)\omega=(\sigma_a-1)(y\sqrt{a}/\sqrt{\beta'})=-1.
		\]
        We have $\omega^2=\alpha\beta'$. Since $F_b$ coincides with the $N'$-invariant subalgebra of $(F_{a,b})_\beta$ and $(F_{a,b})_\alpha$, the isomorphism $(F_{a,b})_{\alpha}\to (F_{a,b})_\beta$ restricts to an automorphism of $F_b$. If this automorphism is the identity, then $\beta'=\beta$ and hence $\omega^2=\alpha\beta$. If this automorphism is non-trivial, then it must be equal to $\sigma_b$, and hence $\beta' = \sigma_b(\beta)$. Define $\omega'\coloneqq (\omega\beta)/(y\sqrt{a})\in F_{a,b}^\times$. Then \[(\omega')^2 =\frac{\omega^2\beta^2}{y^2a}=\frac{\alpha\sigma_b(\beta)\beta^2}{y^2a}=\alpha\beta.\]
        Since $\beta/(y\sqrt{a})$ belongs to $F_{a,b}^\times$, we have $(\sigma_a-1)(\sigma_b-1)(\beta/y\sqrt{a})=1$ and hence \[(\sigma_a-1)(\sigma_b-1)\omega'=(\sigma_a-1)(\sigma_b-1)\omega=-1.\]
        It thus suffices to replace $\omega$ by $\omega'$.

		Conversely, suppose $\alpha\beta=\omega^2$ for some $\omega\in F_{a,b}^\times$ such that $(\sigma_a-1)(\sigma_b-1)\omega=-1$.
		We have 
		\[
		(\sigma_a-1)\omega^2=(\sigma_a-1)\alpha=x^2 b/\alpha^2=(x\sqrt{b}/\alpha)^2.
		\]
  Let \[\zeta_a\coloneqq (\sigma_a-1)\omega \cdot (\alpha/(x\sqrt{b}))\in F_{a,b}^\times.\] 
  Then $(\sigma_b-1)\eta_a=(-1)\cdot(-1)=1$, and hence $\eta_a\in F_a^\times$. Moreover, we have $\eta_a^2=1$ and $N_a(\eta_a)=N_a(\alpha)/(x^2b)=1$. It follows that $\eta_a=\pm 1$, which is equivalent to
  \[(\sigma_a-1)\omega=\pm x\sqrt{b}/\alpha.\]
		Replacing $x$ by $-x$ if necessary (by \Cref{minus-x}, this does not change the $U_3$-algebra up to isomorphism), 
		we may assume that $(\sigma_a-1)\omega=x\sqrt{b}/\alpha$. Similarly, we have $(\sigma_b-1)\omega=y\sqrt{a}/\beta$.
		A calculation shows that the assignment $\sqrt{\alpha}\mapsto \omega/\sqrt{\beta}$ yields an isomorphism of
		$U_3$-algebras $(F_{a,b})_\alpha$ and $(F_{a,b})_\beta$.
	\end{proof}
	
	\subsection{Galois \texorpdfstring{$U_4$}{U4}-algebras}
	
	Let $a,b,c\in F^\times$, and suppose that $(a,b)=(b,c)=0$ in $\on{Br}(F)$. We write $(\Z/2\Z)^3=\ang{\sigma_a,\sigma_b,\sigma_c}$ and view $F_{a,b,c}$ as a Galois $(\Z/2\Z)^3$-algebra over $F$ as in \Cref{kummer-sub}. Let $\epsilon\in F_{a,c}^\times$ satisfy $N_{a,c}(\epsilon)=bx^2$ for some $x\in F^\times$.
	Set $\alpha=N_c(\epsilon)\in F_a^\times$ and $\gamma=N_a(\epsilon)\in F_c^\times$.

    We have 
     \begin{align*}U_4=\langle\sigma_a,\sigma_b,\sigma_c\colon \quad &\sigma_a^2=\sigma_b^2=\sigma_c^2=1,\\
    &[\sigma_a,\sigma_b]^2=[\sigma_b,\sigma_c]^2=[\sigma_a,\sigma_c]=1,\\
    &[[\sigma_a,\sigma_b],\sigma_c]=[\sigma_a,[\sigma_b,\sigma_c]],\\
    &[[\sigma_a,\sigma_b],\sigma_c]^2=1\rangle.
    \end{align*}
	The quotient map $U_4\to (\Z/2\Z)^3$ is given by $\sigma_a\mapsto \sigma_a$, $\sigma_b\mapsto \sigma_b$ and $\sigma_c\mapsto \sigma_c$. 
 The \'etale $F$-algebra $(F_{a,b,c})_{\alpha,\gamma,\epsilon}$ may be given the structure of a Galois $U_4$-algebra as follows: we let $\sigma_a,\sigma_b,\sigma_c$ act on $F_{a,b,c}$ via the quotient $(\Z/2\Z)^3$, and set
	\[
	\sigma_a(\sqrt{\epsilon})=\sqrt{\gamma}/\sqrt{\epsilon}, \quad \quad\sigma_c(\sqrt{\epsilon})=\sqrt{\alpha}/\sqrt{\epsilon}, 
	\]
	\[
	\sigma_a(\sqrt{\alpha})=x\sqrt{b}/\sqrt{\alpha}, \quad \quad \sigma_c(\sqrt{\gamma})=x\sqrt{b}/\sqrt{\gamma},
	\]
	\[
	(\sigma_b-1)(\sqrt{\alpha})=(\sigma_b-1)(\sqrt{\gamma})=(\sigma_b-1)(\sqrt{\epsilon})=(\sigma_a-1)(\sqrt{\gamma})=(\sigma_c-1)(\sqrt{\alpha})=1.
	\]
	We leave it to the reader to verify that this defines a Galois $U_4$-algebra structure on $(F_{a,b,c})_{\alpha,\gamma,\epsilon}$, that is, that $\sigma_a,\sigma_b,\sigma_c$ act via $F$-algebra automorphisms, that these automorphisms satisfy the above relations, and that the subalgebra of $Q_4$-invariants is $F_{a,b,c}$.

	\begin{prop}\label{uu4}
		Let $a,b,c\in F^\times$. 
		Then every Galois $U_4$-algebra $K$ over $F$ such that $K^{Q_4}=F_{a,b,c}$ is of the form $(F_{a,b,c})_{N_c(\epsilon),N_a(\epsilon),\epsilon}$ for some $\epsilon\in F_{a,c}^\times$ with the property
		$N_{a,c}(\epsilon)=b$ in $F^\times/F^{\times 2}$.
	\end{prop}
 
	\begin{proof}
  The proof is similar to that of \Cref{uu3}(a). We have $U_4=N\rtimes S$, where
  \[N=\begin{bmatrix} 
  1 & 0 & * & * \\
  & 1 & * & * \\
  && 1 & 0\\
  &&& 1
  \end{bmatrix},\qquad 
  S=\begin{bmatrix} 
  1 & * & 0 & 0 \\
  & 1 & 0 & 0 \\
  && 1 & *\\
  &&& 1
  \end{bmatrix}.
  \]  
  The coordinate functions $u_{12}$ and $u_{34}$ determine a group isomorphism $S\simeq (\Z/2\Z)^2$.

    Let $\rho\colon \Gamma_F\to U_4$ be a homomorphism whose class in $H^1(F,U_4)$ is equal to the class of $K$, and let $\pi\coloneqq p\circ \rho$, so that the class of $\pi$ in $H^1(F,S)=H^1(F,(\Z/2\Z)^2)$ is equal to $(\chi_a,\chi_b)$. Write $N_{\pi}$ for the twist of $N$. The $S$-module $N$ has a permutation basis given by
  \[
   \begin{bmatrix} 
  1 & 0 & 0 & 0 \\
  & 1 & 1 & 0 \\
  && 1 & 0\\
  &&& 1
  \end{bmatrix},\quad 
     \begin{bmatrix} 
  1 & 0 & 1 & 0 \\
  & 1 & 1 & 0 \\
  && 1 & 0\\
  &&& 1
  \end{bmatrix},\quad 
     \begin{bmatrix} 
  1 & 0 & 0 & 0 \\
  & 1 & 1 & 1 \\
  && 1 & 0\\
  &&& 1
  \end{bmatrix},\quad 
     \begin{bmatrix} 
  1 & 0 & 1 & 1 \\
  & 1 & 1 & 1 \\
  && 1 & 0\\
  &&& 1
  \end{bmatrix}.
  \]
	This choice of basis gives a commutative square of continuous $\Gamma_F$-modules
 \begin{equation}\label{induced-u4}
    \begin{tikzcd}
    N_{\pi} \arrow[r,"\sim"] \arrow[d,"(u_{23})_{\pi}"] & \on{Ind}_{F_{a,c}}^F(\Z/2\Z) \arrow[d,"N_{a,c}"] \\
    \Z/2\Z \arrow[r,equal] & \Z/2\Z. 
    \end{tikzcd}
  \end{equation}
From Faddeev--Shapiro's lemma, we obtain an isomorphism \[\Phi\colon H^1(F,N_{\pi})\to H^1(F_{a,c},N)\xrightarrow{u_{14}}H^1(F_{a,c},\Z/2\Z)\] fitting into the commutative diagram
\begin{equation}\label{shapiro-u4}
\begin{tikzcd}
    H^1(F,N_{\pi}) \arrow[r,"\Phi"] \arrow[d,"(u_{23})_{\pi}"] & H^1(F_{a,c},\Z/2\Z) \arrow[r,equal]\arrow[d,"N_{a,c}"]  & F_{a,c}^\times/F_{a,c}^{\times 2} \arrow[d,"N_{a,c}"] \\
    H^1(F,\Z/2\Z) \arrow[r,equal] & H^1(F,\Z/2\Z) \arrow[r,equal] & F^{\times}/F^{\times 2}.  
\end{tikzcd}
\end{equation}

Let $E$ be a Galois $U_4$-algebra over $F$ and $\varphi\colon E^{Q_4}\to F_{a,b,c}$ be an isomorphism. Base change to $F_{a,c}$ yields an isomorphism $\varphi_{F_{a,c}}\colon (E^N)_{a,c}\to (F_{a,c})_{a,c}=\prod_{\sigma\in S}F_{a,c}$, and so we may write $E_{a,c}=\prod_{\sigma\in S}E_{\varphi,\sigma}$, where $E_{\varphi,\sigma}$ is the subalgebra of $E_{a,c}$ lying over the inverse image of the factor of $\prod_{\sigma\in S}F_{a,c}$ corresponding to $\sigma$. 
Under the identification of (\ref{twist-galois-algebras}), the map $\Phi$ sends the pair $(E,\varphi)$ to the class of the Galois $\Z/2\Z$-algebra $E_{\varphi, 0}/F_{a,c}$, where $0\in S$ is the identity element. 

Since $K^{Q_4}=F_{a,b,c}$, the pair $(K,\on{id})$ defines an element of $H^1(F,N_{\pi})$. Let $\epsilon\in F_{a,c}^\times$ be such that $\Phi$ sends the class of $(K,\on{id})$ to $(F_{a,c})_{\epsilon}$. By (\ref{shapiro-u4}), we have $N_{a,c}(\epsilon)=bx^2$ for some $x\in F^{\times}$. Let $\alpha\coloneqq N_c(\epsilon)\in F_a^\times$ and $\gamma\coloneqq N_a(\epsilon)\in F_c^\times$, and consider on $(F_{a,b,c})_{\alpha,\gamma,\epsilon}$ the Galois $U_4$-algebra structure over $F$ given by (\ref{uu4}). Then $\Phi$ sends the class of the pair $((F_{a,b,c})_{\alpha,\delta,\epsilon},\on{id})$  to the class of the $\Z/2\Z$-algebra $(F_{a,c})_{\epsilon}/F_{a,c}$. Since $\Phi$ is injective, we deduce an isomorphism $K\cong (F_{a,b,c})_{\alpha,\gamma,\epsilon}$ of Galois $U_4$-algebras over $F$.
	\end{proof}

    The following theorem was proved by Hopkins and Wickelgren \cite{hopkins2015splitting} when $F$ is a number field, and by Min\'{a}\v{c} and T\^{a}n \cite{minac2015triple} when $F$ is arbitrary. We give an alternative proof.

	\begin{cor}\label{massey-triple}
	Suppose that $p=2$, and let $a,b,c\in F^\times$ be such that $(a,b)=(b,c)=0$ in $\on{Br}(F)$. Then the Massey product $\ang{a,b,c}$ is defined and vanishes.
\end{cor} 

\begin{proof}
   Since $(a,b)=(b,c)=0$, by \Cref{cup-norm} we have $b\in N_a\cap N_c$.  By \Cref{comes-from-ac-cor} there exists $\epsilon \in F_{a,c}^\times$ such that $N_{a,c}(\epsilon)=b$ in $F^\times/F^{\times 2}$. By the discussion preceding \Cref{uu4}, $K\coloneqq (F_{a,b,c})_{N_c(\epsilon),N_a(\epsilon),\epsilon}$ is a Galois $U_4$-algebra over $F$ such that $K^{Q_4}=F_{a,b,c}$. The conclusion follows from \Cref{dwyer-cor}. 
		\end{proof}
 
	\subsection{Galois \texorpdfstring{$\cl{U}_5$}{U5}-algebras}\label{uu5-sec}
	Let $a,b,c,d\in F^\times$. We write $(\Z/2\Z)^4=\ang{\sigma_a,\sigma_b,\sigma_c,\sigma_d}$ and regard $F_{a,b,c,d}$ as a Galois $(\Z/2\Z)^4$-algebra over $F$ as in \Cref{kummer-sub}.
 
	\begin{prop}\label{uu5}
		Let $a,b,c,d\in F^\times$. Then the Massey product $\langle a,b,c,d\rangle$ is defined if and only if
		there are $\epsilon\in F_{a,c}^\times$, $\nu\in F_{b,d}^\times$ and $\omega\in F_{b,c}^\times$ such that 
		\begin{enumerate}
			\item $N_{a,c}(\epsilon)=b$ in $F^\times/F^{\times 2}$;
			\item $N_{b,d}(\nu)=c$ in $F^\times/F^{\times 2}$;
			\item $N_a(\epsilon) N_d(\nu)=\omega^2$;
			\item $(\sigma_b-1)(\sigma_c-1)\omega=-1$.
		\end{enumerate}
	\end{prop}
	\begin{proof}
	Denote by $U_4^+$ and $U_4^-$ the top-left and bottom-right $4\times 4$ corners of $U_5$, respectively, and let $S\coloneqq U_4^+\cap U_4^-$ be the middle subgroup $U_3$. Let $Q_4^+$ and $Q_4^-$ be the kernel of the maps $U_4^+\to (\Z/2\Z)^3$ and $U_4^-\to (\Z/2\Z)^3$, respectively, and $P_4^+$ and $P_4^-$ be the kernel of the maps $U_4^+\to U_3$ and $U_4^-\to U_3$, respectively.  By \Cref{glueing}, there exists a Galois $\cl{U}_5$-algebra $K/F$ such that $K^{\cl{Q}_5}=F_{a,b,c,d}$ if and only if there are a Galois $U_4^+$-algebra $K_1/F$ and a Galois $U_4^-$-algebra $K_2/F$ such that 
 (i) $K_1^{Q^+_4}=F_{a,b,c}$, $K_2^{Q^-_4}=F_{b,c,d}$ and (ii) the $U_3$-algebras $K_1^{P_4^+}$ and $K_2^{P_4^-}$ are isomorphic. By Proposition \ref{uu4}, to give $K_1$ and $K_2$ satisfying (i) is equivalent to giving two elements
		$\epsilon\in F_{a,c}^\times$ and $\nu\in F_{b,d}^\times$ such that $N_{a,c}(\epsilon)=b$ and $N_{b,d}(\nu)=c$ in $F^\times/F^{\times 2}$. 
		By \Cref{uu3}(c), $K_1$ and $K_2$ satisfy (ii) if and only if there is $\omega\in F_{b,c}^\times$ such that
		$N_a(\epsilon) N_d(\nu)=\omega^2$ and $(\sigma_b-1)(\sigma_c-1)\omega=-1$.
	\end{proof}

\begin{rmk}
    Let $\epsilon\in F_{a,c}^\times$, $\nu\in F_{b,d}^\times$, $e\in F^\times$ and $\omega\in F_{b,c}^\times$ such that $N_a(\epsilon)N_d(\nu)=e\omega^2$. Then $(\sigma_b-1)N_a(\epsilon)=1$ and $(\sigma_c-1)N_d(\nu)=1$, hence $(\sigma_b-1)(\sigma_c-1)\omega^2=1$. Therefore $(\sigma_b-1)(\sigma_c-1)\omega\in\set{\pm 1}$.
 \end{rmk}

  	Suppose now that $a,b,c,d\in F^\times$ satisfy $(a,b)=(b,c)=(c,d)=0$. By \Cref{comes-from-ac-cor} there exist $\epsilon\in F_{a,c}^\times$, $\nu\in F_{b,d}^\times$ such that $N_{a,c}(\epsilon)=b$ in $F^\times/F^{\times 2}$ and $N_{b,d}(\nu)=c$ in $F^\times/F^{\times 2}$. Even if $\ang{a,b,c,d}$ is defined, it is not true that one may find $\omega\in F_{b,c}^\times$ such that $(\epsilon,\nu,\omega)$ satisfies the equations of \Cref{uu5}: one might need to change $\epsilon$ and $\nu$. It will be useful to have a criterion for $\ang{a,b,c,d}$ to be defined in terms of any given $\epsilon$ and $\nu$. This is the content of the next proposition.

		\begin{prop}\label{u5-e}
  Let $a,b,c,d\in F^\times$ be such that $(a,b)=(b,c)=(c,d)=0$.
        Let $\epsilon \in F_{a,c}^\times$ and $\nu\in F_{b,d}^\times$ be such that $N_{a,c}(\epsilon)=b$ in $F^\times/F^{\times 2}$ and $N_{b,d}(\nu)=c$ in $F^\times/F^{\times 2}$.
  
		(a) There exist $e\in F^\times$ and $\omega \in F_{b,c}^\times$ such that
		\[N_a(\epsilon) N_d(\nu)=e\omega^2,\qquad (\sigma_b-1)(\sigma_c-1)\omega=-1.\]
		
		(b) Letting $\epsilon$ and $\nu$ vary, the corresponding $e$ form a $N_aN_{ac}N_dN_{bd}$-coset of $F^\times$. 
		
		(c) The Massey product $\ang{a,b,c,d}$ is defined if and only if $e\in N_aN_{ac}N_dN_{bd}$. 
	\end{prop}
	
	\begin{proof}
	(a) By \Cref{differ-by-scalar}, there exists $e\in F^\times$ such that $(F_{b,c})_{N_a(\epsilon)}\simeq (F_{b,c})_{eN_d(\nu)}$. Thus \Cref{uu3}(c) implies the existence of $\omega\in F_{b,c}^\times$ such that \[N_a(\epsilon) N_d(\nu)=e\omega^2,\qquad (\sigma_b-1)(\sigma_c-1)\omega=-1.\]
	
	(b)  We first show that any two values of $e$ differ by an element of $N_aN_{ac}N_dN_{bd}$. For this, we suppose given $\epsilon\in F_{a,c}^\times$, $\nu\in F_{b,d}^\times$, $x,y,e\in F^\times$ and $\omega\in F_{b,c}^\times$ such that $N_{a,c}(\epsilon)=x^2$, $N_{b,d}(\nu)=y^2$, $N_a(\epsilon)N_d(\nu)=e\omega^2$, and we prove that $e\in N_aN_{ac}N_dN_{bd}$. (We could also assume that $(\sigma_b-1)(\sigma_c-1)\omega=1$, but we will see that it follows from the rest.)
	
	By \Cref{biquadratic-triple-norm}(1), there exist $\epsilon_a\in F_a^\times$, $\epsilon_c\in F_c^\times$ and $\epsilon_{ac}\in F_{ac}^\times$ such that $\epsilon=\epsilon_a\epsilon_c\epsilon_{ac}$, as well as $\nu_b\in F_b^\times$, $\nu_d\in F_d^\times$ and $\nu_{bd}\in F_{bd}^\times$ such that $\nu=\nu_b\nu_d\nu_{bd}$. 
 
 We have 
 \[N_a(\epsilon)=N_a(\epsilon_a)N_{ac}(\epsilon_{ac})\epsilon_c^2,\qquad N_d(\nu)=N_d(\nu_d)N_{bd}(\nu_{bd})\nu_b^2.\]
 Define $\omega_1\coloneqq \omega/(\epsilon_c\nu_b)\in F_{b,c}^\times$. Then $N_a(\epsilon)N_d(\nu)=e\omega^2$ may be rewritten as
 \begin{equation}\label{norms-omega1}N_a(\epsilon_a)N_{ac}(\epsilon_{ac})N_d(\nu_d)N_{bd}(\nu_{bd})=e\omega_1^2.\end{equation}
 In particular, $\omega_1^2$ belongs to $F^\times$, hence $\omega_1$ belongs to at least one of $F_b^\times$, $F_c^\times$ and $F_{bc}^\times$. A simple computation now shows that $\omega_1=f\sqrt{b}^i\sqrt{c}^j$, where $f\in F^\times$ and $i,j\in \set{0,1}$. Since $b=(-d)\cdot(-bd)\cdot d^{-2}\in N_dN_{bd}$ and $c=(-a)\cdot(-ac)\cdot c^{-2}\in N_aN_{ac}$, we deduce that $\omega_1^2\in N_aN_{ac}N_dN_{bd}$. Now (\ref{norms-omega1}) implies that $e\in N_aN_{ac}N_dN_{bd}$, as desired.
 
	 For the converse, suppose that $e=N_a(\epsilon_a) N_{ac}(\epsilon_{ac}) N_d(\nu_d) N_{bd}(\nu_{bd})$, where $\epsilon_a\in F_a^\times$, $\epsilon_{ac}\in F_{ac}^\times$, $\nu_d\in F_d^\times$ and $\nu_{bd}\in F_{bd}^\times$. Set $\epsilon=\epsilon_a \epsilon_{ac}$, $\nu=\nu_d \nu_{bd}$ and $\omega=1$. Then $N_{a,c}(\epsilon)\in F^{\times 2}$, $N_{b,d}(\nu)\in F^{\times 2}$, $N_a(\epsilon)N_d(\nu)=e\omega^2$ and $(\sigma_b-1)(\sigma_c-1)\omega=1$, as desired. 
	
	(c) This follows from (b) and \Cref{uu5}.
	\end{proof}

\subsection{Splitting varieties}\label{splitting-varieties-sec}
We now interpret \Cref{uu5} in terms of splitting varieties. The material of this section is not needed for the proofs of Theorems \ref{massey-deg}, \ref{massey-double-deg} and \ref{mainthm-positselski}.

Let $n\geq 2$ be an integer, $a_1,\dots,a_n\in F^\times$, and $V$ be an $F$-variety. Consider the following property: For all field extensions $K/F$ we have
   \begin{equation}\label{splitting-eq}
   \text{$V(K)\neq \emptyset$ $\iff$ $\ang{a_1,\dots,a_n}$ vanishes over $K$.}
   \end{equation}
   In the literature, a variety $V$ satisfying (\ref{splitting-eq}) is sometimes called a \emph{splitting variety} for $\ang{a_1,\dots,a_n}$. 

   The geometry of splitting varieties becomes increasingly complicated as $n$ gets bigger. When $n=2$, a splitting variety for $\ang{a_1,a_2}$ is the $F$-conic corresponding to the symbol $(a_1,a_2)$. Hopkins and Wickelgren \cite{hopkins2015splitting} constructed a splitting variety for $n=3$: it is a torsor under a torus. When $n=4$, a splitting variety was obtained in \cite{guillot2018fourfold}. 
   P\'al and Schlank  \cite{pal2022brauer} constructed splitting varieties for all $n$: their examples are homogeneous spaces under $\on{SL}_n$ with finite supersolvable stabilizers. These varieties were exploited by \cite{harpaz2019massey} for the proof of \Cref{massey-conj} when $F$ is a number field. 

Let $a,b,c,d\in F^\times$, and consider the $F$-torus 
    \[S\coloneqq R_{a,c}(\G_{\on{m}})\times R_{b,d}(\G_{\on{m}})\times \G_{\on{m}}^2\times R_{b,c}(\G_{\on{m}}),\]
whose coordinates we denote by $(\epsilon,\nu,x,y,\omega)$.
Let $T\subset S$ be the $F$-subgroup defined by the equations
\begin{enumerate}
    \item $(\sigma_c-1)\omega=x/N_a(\epsilon)$;
    \item $(\sigma_b-1)\omega=y/N_d(\nu)$;
    \item $N_a(\epsilon)N_d(\nu)=\omega^2$.
\end{enumerate}
Lattice computations show that $T$ is a torus. Consider the $T$-torsor $X\subset S$ given by the equations
\begin{itemize}
    \item[(1')] $(\sigma_c-1)\omega=x\sqrt{b}/N_a(\epsilon)$;
    \item[(2')] $(\sigma_b-1)\omega=y\sqrt{c}/N_d(\nu)$;
    \item[(3')] $N_a(\epsilon)N_d(\nu)=\omega^2$.
\end{itemize}

We now show that $X$ satisfies a variant of (\ref{splitting-eq}), where ``vanishes'' is replaced by ``is defined.''

\begin{prop}\label{generic-var}
    For all field extensions $K/F$, we have $X(K)\neq \emptyset$ if and only if $\ang{a,b,c,d}$ is defined over $K$.
\end{prop}

\begin{proof}
    Since the formation of $T$ and $X$ commutes with arbitrary field extensions, we may suppose that $K=F$. If $(\epsilon,\nu,x,y,\omega)$ belongs to $X(F)$, then $N_{a,c}(\epsilon)=bx^2$, $N_{b,d}(\nu)=cy^2$ and    
    $(\sigma_b-1)(\sigma_c-1)\omega=(\sigma_b-1)(x\sqrt{b}/N_a(\epsilon))=-1$, hence $\ang{a,b,c,d}$ is defined by \Cref{uu5}.

    Conversely, suppose that $\ang{a,b,c,d}$ is defined. By \Cref{uu5}, there exist $\epsilon\in F_{a,c}^\times$, $\nu\in F_{b,d}^\times$ and $\omega\in F_{b,c}^\times$ such that 
    \[N_a(\epsilon)N_d(\nu)=\omega^2,\qquad (\sigma_b-1)(\sigma_c-1)\omega=-1.\] Define $x,y\in F_{b,c}^\times$ by
    \[x\coloneqq \frac{(\sigma_c-1)\omega\cdot N_a(\epsilon)}{\sqrt{b}},\qquad y\coloneqq \frac{(\sigma_b-1)\omega\cdot N_d(\nu)}{\sqrt{c}}.\]
			Then 
   \[(\sigma_b-1)x=(\sigma_b-1)(\sigma_c-1)\omega\cdot(\sigma_b-1)(N_a(\epsilon))\cdot (1-\sigma_b)\sqrt{b}=(-1)\cdot 1\cdot (-1)=1.\]
	Moreover, since $N_d(\nu)\in F_b^\times$, we have $(\sigma_c-1)(N_a(\epsilon))=(\sigma_c-1)(\omega^2)$, therefore
 \[(\sigma_c-1)x=(\sigma_c-1)^2(\omega)\cdot (\sigma_c-1)(N_a(\epsilon))=(2-2\sigma_c)(\omega)\cdot(\sigma_c-1)(\omega^2)=1.\]
			It follows that $x$ belongs to $F^\times$. Similar calculations show that $y$ belongs to $F^\times$, and hence $(\epsilon,\nu,x,y,\omega)$ belongs to $X(F)$. This completes the proof.
\end{proof}

\begin{cor}\label{odd-defined}
    Let $a,b,c,d\in F^\times$. Then $\ang{a,b,c,d}$ is defined over $F$ if and only if there exists a finite field extension of odd degree $F'/F$ such that $\ang{a,b,c,d}$ is defined over $F'$.
\end{cor}

\begin{proof}
    The Massey product $\ang{a,b,c,d}$ vanishes if $a,b,c,d$ are all squares (this is immediate for example from \Cref{dwyer}(iv)), hence $\ang{a,b,c,d}$ vanishes over $F_{a,b,c,d}$. By \Cref{generic-var}, this implies that $X(F_{a,b,c,d})\neq\emptyset$,  that is, the $T$-torsor $X$ is split by $F_{a,b,c,d}/F$. Thus, by a restriction-corestriction argument, the order of $[X]\in H^1(F,T)$ is a power of $2$. Therefore, if $X$ is split by an extension of odd degree, the order of $[X]$ in $H^1(F,T)$ is odd and a power of two, hence $1$.
\end{proof}

\begin{rmk}
    The variety $X$ is a torsor under a torus. In contrast, all known splitting varieties for $n=4$ are quite involved. In particular, while \Cref{massey-conj} predicts that \Cref{odd-defined} should also be true if ``defined'' is replaced by ``vanishing,'' we do not know how to prove it.
\end{rmk}

\subsection{Galois \texorpdfstring{$U_5$}{U5}-algebras} 

Let $a,b,c,d\in F^\times$. In \cite[Theorem A]{guillot2018fourfold}, an equivalent condition for the vanishing of the Massey product $\ang{a,b,c,d}$ was given. In this section, we recover this result using our methods, and then we specialize to the case $a=d$. Our proof and that of \cite[Theorem A]{guillot2018fourfold} are closely related. In particular the short exact sequence (\ref{exactexact}) below has been used in the proof of \cite[Theorem A]{guillot2018fourfold}; see \cite[\S 2.4 and Proof of Theorem 3.3]{guillot2018fourfold}.

   \begin{prop}\label{u5}
  Let $a,b,c,d\in F^\times$. The Massey product $\langle a,b,c,d\rangle$ vanishes if and only if there exist $\alpha\in F_a^\times$ and $\delta\in F_d^\times$
such that
\begin{enumerate}
			\item $N_{a}(\alpha)=b$ in $F^\times/F^{\times 2}$;
			\item $N_{d}(\delta)=c$ in $F^\times/F^{\times 2}$;
			\item $(\alpha,\delta)=0$ in $\on{Br}(F_{a,d})$.
			\end{enumerate}
\end{prop}

\begin{proof}
Write $P$ for the subgroup of $U_5$ defined by $u_{12}=u_{13}=u_{23}=u_{34}=u_{35}=u_{45}=0$.
This is an abelian normal subgroup of $U_5$. There is an exact sequence
\begin{equation}\label{exactexact}
1\to P \to U_5 \to U_3\times U_3 \to 1,
\end{equation}
so $P$ has a natural structure of a $(U_3\times U_3)$-module.

Let $N$ and $S$ be the subgroups of $U_3$ as in the proof of Proposition \ref{uu3}(a).
In particular, $N$ is an $S$-module (by conjugation). Let
$N'$ and $S'$ be the corresponding subgroups of $U_3$ as in the proof of Proposition \ref{uu3}(b).

The bilinear map
\[ N\times N'\to P \]
taking a pair of matrices
\[
		\begin{bmatrix}
			1 & 0 & f_1  \\
			& 1 & e_1  \\
			&    & 1
		\end{bmatrix},
		\ \ \ \begin{bmatrix}
			1 & e_2 & f_2  \\
			& 1 & 0  \\
			&    & 1
		\end{bmatrix}
		\]
to
\[
        \begin{bmatrix}
			1 & 0 & 0 & f_1 e_2 & f_1 f_2 \\
			& 1 & 0 & e_1 e_2 & e_1 f_2 \\
			&   & 1 & 0 & 0 \\
            &   &   & 1 & 0 \\
            &   &   &   & 1
		\end{bmatrix}
		\]
yields an isomorphism of $(U_3\times U_3)$-modules
\[
N\otimes N'\xrightarrow{\sim} P.
\]

The natural projections $t:U_3\to N$ and $t':U_3\to N'$ are  $1$-cocycles.
A direct calculation shows that the class in $H^2(U_3\times U_3,P)\simeq H^2(U_3\times U_3,N\otimes N')$ of the exact sequence (\ref{exactexact})
is equal to the cup-product $t\cup t'$.

Let $\alpha\in F_a^\times$ be such that $N_a(\alpha)=b\in F^\times/F^{\times 2}$ and let $h:\Gamma_F\to U_3$ be a
group homomorphism corresponding to the Galois $U_3$-algebra $(F_{a,b})_\alpha$ via (\ref{galois-alg}). Similarly, let 
$\delta\in F_d^\times$ be such that $N_d(\delta)=c\in F^\times/F^{\times 2}$ and let $h':\Gamma_F\to U_3$ be a group homomorphism corresponding to $(F_{c,d})_{\delta}$ via (\ref{galois-alg}). 

As in the proof of Proposition \ref{uu3}, the $\Gamma_F$-module $N$, where $\Gamma_F$ acts via $h$, is the induced module $\on{Ind}_{F_a}^{F}(\Z/2\Z)$. It follows that the image of $t$ 
under the composition
\[
H^1(U_3,N)\xrightarrow{h^*} H^1(F,N)=H^1(F_a,\Z/2\Z)=F_a^\times/F_a^{\times 2}
\]
is equal to the class of $\alpha$. Similarly, the image of $t'$ 
under the composition
\[
H^1(U_3,N')\xrightarrow{(h')^*} H^1(F,N')=H^1(F_d,\Z/2\Z)=F_d^\times/F_d^{\times 2}
\]
is equal to the class of $\delta$.

Note that
\[
P=N\otimes N'=\on{Ind}_{F_{a,d}}^{F}(\Z/2\Z),
\]
where we view $P$ as a $\Gamma_F$-module via $(h,h'):\Gamma_F\to U_3\times U_3$.

Consider the commutative diagram
\[
		\begin{tikzcd}
			H^1(U_3,N)\otimes H^1(U_3,N') \arrow[d,"h^*\otimes (h')^*"] \arrow[r,"\cup"] & H^2(U_3 \times U_3,P) \arrow[d,"\text{$(h,h')^*$}"] \\
			H^1(F,N)\otimes H^1(F,N')  \arrow[d,equal] \arrow[r,"\cup"] & H^2(F,P) \arrow[d,equal] \\	
            H^1(F_a,\Z/2\Z)\otimes H^1(F_d,\Z/2\Z)  \arrow[r,"\cup"] & H^2(F_{a,d},\Z/2\Z).
		\end{tikzcd}
		\]
It follows that the image of $t\otimes t'$ under the composition
\[H^2(U_3\times U_3, P)\xrightarrow{(h,h')^*} H^2(F,P)=H^2(F_{a,d},\Z/2\Z)\subset\on{Br}(F_{a,d})\]
is equal to the cup-product $(\alpha,\delta)$. The homomorphism $(h,h'):\Gamma_F\to U_3\times U_3$ lifts to a homomorphism $\Gamma_F\to U_5$ if and only if the pullback of the exact sequence (\ref{exactexact}) via $(h,h')$ is split, that is, if and only if the image of $t\otimes t'$
in $H^2(F,P)=H^2(F_{a,d},\Z/2\Z)$ is trivial. Since this image is $(\alpha,\delta)\in H^2(F_{a,d},\Z/2\Z)\subset\on{Br}(F_{a,d})$, this and \Cref{dwyer}(iv) imply the conclusion.
\end{proof}

The following result is a reformulation of \cite[Theorem A]{guillot2018fourfold}. 

\begin{cor}\label{u5-rephrase}
    Let $a,b,c,d\in F^\times$ be such that $(a,b)=(c,d)=0$ in $\on{Br}(F)$. Let $\alpha\in F_a^\times$ and $\delta\in F_d^\times$ be such that $N_{a}(\alpha)=b$ in $F^\times/F^{\times 2}$ and $N_{d}(\delta)=c$ in $F^\times/F^{\times 2}$. The Massey product $\langle a,b,c,d\rangle$ vanishes if and only if there exist $x,y\in F^\times$ such that $(\alpha x,\delta y)=0$ in $\on{Br}(F_{a,d})$.
\end{cor}

\begin{proof}
    Recall that $\alpha$ and $\delta$ exist by \Cref{cup-norm}. Suppose $(\alpha x, \delta y)=0$ in $\on{Br}(F_{a,d})$ for some $x,y\in F^\times$. Since $N_a(\alpha x)=b$ in $F^\times/F^{\times 2}$ and $N_d(\delta y)=c$ in $F^\times /F^{\times 2}$, the Massey product $\ang{a,b,c,d}$ vanishes by \Cref{u5}. 
    
    Conversely, if $\ang{a,b,c,d}$ vanishes, then \Cref{u5} gives $\alpha'\in F_a^\times$ and $\delta'\in F_d^\times$ such that $N_a(\alpha')=b$ in $F^\times/F^{\times 2}$, $N_d(\delta')=c$ in $F^\times/F^{\times 2}$ and $(\alpha',\delta')=0$ in $\on{Br}(F_{a,d})$. There exist $u_a\in F^\times$ and $u_d\in F^\times$ such that $N_a(\alpha')=N_a(\alpha)u_a^2=N_a(\alpha u_a)$ and $N_d(\delta')=N_d(\delta)u_d^2=N_d(\delta u_d)$. Now Hilbert's Theorem~90 implies the existence of $\eta_a\in F_a^\times$ such that $\alpha'=\alpha u_a (\sigma_a-1)\eta_a=\alpha u_aN_a(\eta_a)\eta_a^{-2}$. Similarly, there exists $\eta_d\in F_d^\times$ such that $\delta'=\delta u_d N_d(\eta_d)\eta_d^{-2}$. Set $x\coloneqq N_a(\eta_a)u_a\in F^\times$ and $y\coloneqq N_d(\eta_d)u_d\in F^\times$. Then
    \[0=(\alpha',\delta')=(\alpha x \eta_a^{-2},\delta y \eta_d^{-2})=(\alpha x,\delta y)\qquad \text{in $\on{Br}(F_{a,d})$.}\]
    as desired.
    \end{proof}

\begin{cor}\label{u5-cor}
			Let $a,b,c\in F^\times$ be such that $(a,b)=(c,a)=0$ in $\on{Br}(F)$. Let $\alpha,\delta\in F_a^\times$ be such that $N_a(\alpha)=b$ and $N_a(\delta)=c$. The Massey product $\ang{a,b,c,a}$ vanishes over $F$ if and only if there exist $x,y\in F^\times$ such that $(\alpha x, \delta y)=(\alpha x,c)=0$ in $\on{Br}(F_a)$.
\end{cor}

  \begin{proof}
			Write $F_a=F[u_a]/(u_a^2-a)$ and $F_{a,a}=F[v_a,w_a]/(v_a^2-a,w_a^2-a)$. We have an $F$-algebra isomorphism
			\[\phi\colon F_{a,a}\xrightarrow{\sim} F_a\times F_a,\qquad v_a\mapsto (u_a,u_a), \quad w_a\mapsto (u_a,-u_a).\]
		      If $\pi=\pi_1+\pi_2 v_a+\pi_3 w_a +\pi_4 v_aw_a\in F_{a,a}$, then
			\begin{equation}\label{phi-general-element}\phi(\pi)=(\pi_1+a\pi_4+(\pi_2+\pi_3)u_a,\pi_1-a\pi_4+(\pi_2-\pi_3)u_a).\end{equation}
   Let $x,y\in F^\times$. Since $\alpha$ is in the $F$-span of $1$ and $v_a$, and $\delta$ is in the $F$-span of $1$ and $w_a$, by (\ref{phi-general-element}) we have $\phi(\alpha x)=(\alpha x,\alpha x)$ and $\phi(\delta y)=(\delta y, \sigma_a(\delta)y)$. It follows that, letting \[\varphi_*\colon \on{Br}(F_{a,a})\xrightarrow{\sim} \on{Br}(F_a)\times \on{Br}(F_a)\]
   be the isomorphism induced by $\varphi$, we have 
   \[\varphi_*((\alpha x, \delta y))=((\alpha x, \delta y),(\alpha x, \sigma_a(\delta) y)).\] 
    In particular, $(\alpha x, \delta y)=0$ in  $\on{Br}(F_{a,a})$ if and only if $(\alpha x, \delta y)=(\alpha x, \sigma_a(\delta) y)=0$ in $\on{Br}(F_a)$. Since \[(\alpha x, \delta y)+(\alpha x, \sigma_a(\delta) y)=(\alpha x, c)\qquad \text{in $\on{Br}(F_a)$},\] we deduce that $(\alpha x,\delta y)=0$ in $\on{Br}(F_{a,a})$ if and only if $(\alpha x,\delta y)=(\alpha x,c)=0$ in $\on{Br}(F_a)$. The conclusion follows from \Cref{u5-rephrase}.
		\end{proof}

		\section{Proof of Theorem \ref{massey-deg}}\label{massey-deg-section}
		Let $a,b,c,d\in F^{\times}$, suppose that $b+c=1$ and let $v_1,v_2,u_1,u_2\in F$ be such that
		\begin{equation}\label{vi-ui-def}
			\begin{cases}
				v_1^2-bv_2^2=a,\\
				u_1^2-cu_2^2=d,\\
				v_1v_2u_1u_2(v_1+v_2)(u_1+u_2)(v_1+u_1)\neq 0.
			\end{cases}
		\end{equation}
  By \Cref{cup-norm}, this implies that $(a,b)=(b,c)=(c,d)=0$ in $\on{Br}(F)$.

		Define $r,s,t\in F^\times$ as follows:
		\begin{align*}
			r\coloneqq &2(v_1+v_2)(u_1+u_2)v_2u_2,\\
			s\coloneqq &2(v_1+u_1)(u_1+u_2),\\
			t\coloneqq &2(v_1+u_1)(v_1+v_2).
		\end{align*}	 
		
		As we will explain in \Cref{conclude-proof}, the proof of \Cref{massey-deg} will follow from the next two propositions.

		\begin{prop}\label{lambda-criterion}
			Suppose that $a=d$.
			
			(a) The Massey product $\ang{a,b,c,a}$ is defined over $F$ if and only if $r\in N_aN_{ab}N_{ac}$.
			
			(b) The Massey product $\ang{a,b,c,a}$ vanishes over $F$ if and only if $t\in N_cN_{ac}N_{bc}$.
		\end{prop}
		
		\begin{prop}\label{q-t-s}
			Suppose that $a=d$. Then
			\[r\in N_aN_{ab}N_{ac}\iff s\in N_bN_{ab}N_{bc} \iff t\in N_cN_{ac}N_{bc}.\]
		\end{prop}
	
		\subsection{Proof of Proposition \ref{lambda-criterion}(a)}
		
		We maintain the notations and assumptions of the beginning of \Cref{massey-deg-section}. Since $(a,b)=(b,c)=(c,d)=0$ in $\on{Br}(F)$, by \Cref{massey-triple} the Massey products $\ang{a,b,c}$ and $\ang{b,c,d}$ vanish. Therefore, by \Cref{uu4}, there exist $\epsilon \in F_{a,c}^\times$ and $\nu\in F_{b,d}^\times$ such that $K_1=(F_{a,b,c})_{N_c(\epsilon),N_a(\epsilon),\epsilon}$ and $K_2=(F_{b,c,d})_{N_d(\nu),N_b(\nu),\nu}$ are Galois $U_4$-algebras such that $K_1^{Q_4}=F_{a,b,c}$ and $K_2^{Q_4}=F_{b,c,d}$. By \Cref{u5-e}(a), there exist $e \in F^\times$ and $\omega\in F_{b,c}^{\times}$ such that 
		\[N_a(\epsilon) N_d(\nu)=e\omega^2,\qquad (\sigma_b-1)(\sigma_c-1)\omega=-1.\]

		\begin{prop}\label{lambda-formula}
			We have $e = r$ in $F^\times/N_aN_{ac}N_dN_{bd}$. In particular, $\ang{a,b,c,d}$ is defined if and only if $r\in N_aN_{ac}N_dN_{bd}$.
		\end{prop}

		\begin{proof}
			(a) Define
			\begin{align*}
				\alpha&\coloneqq \frac{v_1}{v_2}+\frac{1}{v_2}\sqrt{a}\in F_a^{\times}\\
				\beta&\coloneqq 1+\sqrt{b}\in F_b^{\times}\\
				\gamma&\coloneqq 1+\sqrt{c}\in F_c^{\times}\\
				\delta&\coloneqq \frac{u_1}{u_2}+\frac{1}{u_2}\sqrt{d}\in F_d^{\times}
			\end{align*}
			Note that $N_a(\alpha)=b=N_c(\gamma)$ and $N_b(\beta)=c=N_d(\delta)$. Set \[\epsilon\coloneqq \alpha+\gamma\in F_{a,c}^{\times},\qquad \nu\coloneqq \beta+\delta\in F_{b,d}^{\times}.\]
			By \Cref{comes-from-ac}(1), we have
			\begin{equation*}
				N_a(\epsilon)=\gamma x,\qquad N_d(\nu)=\beta y,
			\end{equation*}
			where
			\begin{align*}
				x\coloneqq \on{Tr}_a(\alpha)+\on{Tr}_c(\gamma)=2\left(\frac{v_1}{v_2}+1\right)\in F^\times,\\
				y\coloneqq \on{Tr}_b(\beta)+\on{Tr}_d(\delta)=2\left(1+\frac{u_1}{u_2}\right)\in F^\times.
			\end{align*}
			In particular
			\begin{equation*}
				N_{a,c}(\epsilon)=bx^2,\qquad N_{b,d}(\nu)=cy^2. 
			\end{equation*}
			  Define
			\[\omega\coloneqq \frac{1+\sqrt{b}+\sqrt{c}}{v_2u_2}\in F_{b,c}^{\times}.\]
			Note that $\omega\neq 0$ because $1$, $\sqrt{b}$ and $\sqrt{c}$ are linearly independent over $F$. Moreover
        \begin{align*}
            (\sigma_b-1)(\sigma_c-1)\omega&=\frac{(\sigma_b\sigma_c+1)\omega}{(\sigma_b+\sigma_c)\omega}\\
            &=\frac{(1+\sqrt{b}+\sqrt{c})(1-\sqrt{b}-\sqrt{c})}{(1-\sqrt{b}+\sqrt{c})(1+\sqrt{b}-\sqrt{c})}\\
            &=\frac{-2\sqrt{bc}}{2\sqrt{bc}}\\
            &=-1.
        \end{align*}
We have $(1+\sqrt{b}+\sqrt{c})^2=2(1+\sqrt{b})(1+\sqrt{c})$, hence
\begin{align*}
    \frac{N_a(\epsilon)N_d(\nu)}{\omega^2}&=\frac{xy(1+\sqrt{b})(1+\sqrt{c})v_2^2u_2^2}{(1+\sqrt{b}+\sqrt{c})^2}\\
    &=\frac{xyv_2^2u_2^2}{2}\\
    &=2\left(\frac{v_1}{v_2}+1\right)\left(1+\frac{u_1}{u_2}\right)v_2^2u_2^2\\
    &=r.
\end{align*}
Thus $N_a(\epsilon)N_d(\nu)=r\omega^2$. We conclude from \Cref{u5-e}(b) that $e=r$ modulo $N_aN_{ac}N_dN_{bd}$.
		\end{proof}
		
		\begin{proof}[Proof of \Cref{lambda-criterion}(a)]
			Since $a=d$, we have $N_aN_{ac}N_dN_{bd}=N_aN_{ab}N_{ac}$. The conclusion follows from \Cref{lambda-formula} and \Cref{u5-e}. 
		\end{proof}
		
		\subsection{Proof of Proposition \ref{lambda-criterion}(b)}

		The next proposition is the key step for our proof of \Cref{lambda-criterion}(b). Its proof is the only place where we need to use quadratic form theory in this paper.	
		
		\begin{prop}\label{albert}
			Let $a\in F^\times$ and $\pi,\mu\in F_a^\times$ be such that $N_a(\pi,\mu)=0$ in $\on{Br}(F)$. Then there exists $z\in F^\times$ such that $(\pi,\mu z)=0$ in $\on{Br}(F_a)$.
		\end{prop}

		\begin{proof}
			We use the theory of Albert forms attached to biquaternion algebras; see \cite[\S 16 A]{knus1998book}. As explained in \cite[Example (16.4)]{knus1998book}, given a biquaternion $F$-algebra $A\coloneqq (a_1,b_1)\otimes (a_2,b_2)$, the quadratic form $\ang{a_1,b_1,-a_1b_1,-a_2,-b_2,a_2b_2}$ is an Albert form of $A$. Given two presentations of $A$ as a tensor product of two quaternion algebras, the corresponding Albert forms are similar to each other; see \cite[Proposition (16.3)]{knus1998book}. 
			
			Let $K/F$ be an \'etale algebra of degree $2$, let $s\colon K\to F$ be a nonzero linear map such that $s(1)=0$, let $Q$ be a quaternion algebra over $K$, let $Q^0\subset Q$ be the subspace of pure quaternions, let $q:Q^0\to K$ be the quadratic form given by squaring, and let $s_*(q)$ be the transfer of $q$; see \cite[Chapter VII, \S 1]{lam2005introduction}. Then it follows from \cite[Propositions (16.23) and (16.27)]{knus1998book} that $s_*(q)$ is similar to an Albert form over $F$ of the biquaternion $F$-algebra given by the corestriction $N_{K/F}(Q)$; see the proof of \cite[Corollary (16.28)]{knus1998book}. Thus, by Albert's theorem \cite[Theorem 16.5]{knus1998book}:
			\begin{equation}\label{albert-theorem}
				\text{If $N_{K/F}(Q)$ is split then $s_*(q)$ is hyperbolic.}
			\end{equation}
			
			Now let $K= F_a$, let $s\colon F_a \to F$ be a non-zero $F$-linear map such that $s(1) = 0$, and let $Q= (\pi,\mu)$. Then $q=\ang{\pi,\mu,-\pi\mu}$. By assumption, $N_a(\pi, \mu)$ is split, hence by (\ref{albert-theorem}) the $6$-dimensional quadratic form $s_*(q)$ is hyperbolic. Since $4>6/2$, the $4$-dimensional subform $s_*\ang{\mu,-\pi\mu}$ of $s_*(q)$ is  isotropic. We deduce that the form $\ang{\mu,-\pi\mu}$ over $F_a$ represents an element of $F$. If $\ang{\mu,-\pi\mu}$ is isotropic, then $\pi\in F_a^{\times 2}$, hence $(\pi,\mu)=0$ in $\on{Br}(F_a)$ and we may take $z=1$. Otherwise $\ang{\mu,-\pi\mu}$ over $F_a$ represents an element $z\in F^\times$, then $\mu z$ is represented by $\ang{1, -\pi}$. By \Cref{cup-norm}, this implies that $(\pi,\mu z)=0$ in $\on{Br}(F_a)$ and completes the proof.
		\end{proof}

We maintain the notations and assumptions of the beginning of \Cref{massey-deg-section}. Suppose further that $a=d$. Let 
		\begin{equation}\label{define-l}l\coloneqq v_1+u_1,\qquad \alpha\coloneqq lv_1+l\sqrt{a}\in F_a^{\times},\qquad \delta\coloneqq lu_1+l\sqrt{a}\in F_a^{\times}.\end{equation}
		By (\ref{vi-ui-def}), we have
		\[N_{a}(\alpha)=b(lv_2)^2,\qquad N_{a}(\delta)=c(lu_2)^2.\]

		\begin{cor}\label{albert-cor}
			(a) If $N_a(\alpha x,\delta y)=0$ in $\on{Br}(F)$ for some $x,y\in F^{\times}$, then $x\in N_cN_{bc}$.
			
			(b) For every $x\in N_cN_{bc}$, there exists $y\in F^\times$ such that $(\alpha x, \delta y)=0$ in $\on{Br}(F_a)$.
		\end{cor}
		
		\begin{proof}
    We first show that for all $x,y\in F^\times$:
			\begin{equation}\label{albert-eq1}N_a(\alpha x,\delta y)=(b,y)+(x,c)\qquad \text{in $\on{Br}(F)$.}\end{equation} Indeed, since $\alpha+\sigma_a(\delta)=l^2$, by \Cref{cup-norm} we have $(\alpha,\sigma_a(\delta))=0$ in $\on{Br}(F_a)$, hence 
			\[(\alpha,\delta)=(\alpha,\delta)+(\alpha,\sigma_a(\delta))=(\alpha,N_a(\delta))=(\alpha,c)\qquad \text{in $\on{Br}(F_a)$.}\]
			It follows that
			\[(\alpha x,\delta y)=(\alpha,\delta)+(\alpha,y)+(x,\delta)+(x,y)=(\alpha, c)+(\alpha,y)+(x,\delta)+(x,y)\quad \text{in $\on{Br}(F_a)$.}\]
			Now (\ref{albert-eq1}) follows by applying $N_a$ and using that $N_a(\alpha)=b$, $N_a(\delta)=c$, $(b,c)=0$ and $N_a(x,y)=2(x,y)=0$.
  
			(a) By (\ref{albert-eq1}) we have $(b,y)+(x,c)=0$. The conclusion follows from \Cref{chain-lemma}.
			
			(b) Write $x=n_cn_{bc}$, where $n_c\in N_c$  and $n_{bc}\in N_{bc}$. Then by (\ref{albert-eq1}) we have $N_a(\alpha x, \delta n_{bc})=0$.   By \Cref{albert} applied to $\pi=\alpha x$ and $\mu=\delta n_{bc}$, there exists $z\in F^\times$ such that $(\alpha x, \delta n_{bc} z)=0$ in $\on{Br}(F_a)$. Letting $y\coloneqq n_{bc} z$, we obtain $(\alpha x, \delta y)=0$ in $\on{Br}(F_a)$, as desired.
		\end{proof}
		
		\begin{rmk}
			We give a proof of \Cref{albert-cor}(b) that minimizes the use of quadratic form theory. Let $x\in F^\times$ and consider the quadratic form $q_x\coloneqq \ang{\delta,-\alpha x \delta}$ over $F_a$. We first show that:
   \begin{equation}\label{albert-cor-eq}
       \text{$(\alpha x,\delta y)=0$ for some $y\in F^\times$ $\iff$ $q_x$ represents an element of $F^\times$.}
   \end{equation}
   Indeed, let $y\in F^{\times}$. Then $(\alpha x,\delta y)=0$ in $\on{Br}(F_a)$ if and only the quadratic form $\ang{1,-\alpha x}$ over $F_a$ represents $\delta y$. This is in turn equivalent to $y$ being represented by $q_x$. This proves (\ref{albert-cor-eq}).
   
   Now let $s\colon F_a\to F$ be the $F$-linear map such that $s(1)=0$ and $s(\sqrt{a})=1$. The form $q_x$ over $F_a$ represents a value in $F$ if and only if the form $s_*(q_x)$ over $F$ is isotropic. 

    Suppose first that $q_x$ is isotropic. Then $\alpha x \in F_a^{\times 2}$, hence $(\alpha x, \delta)=0$ in $\on{Br}(F_a)$, that is, the conclusion of \Cref{albert-cor}(b) is true for $y=1$. 
   
   Suppose now that $q_x$ is anisotropic. A simple computation shows that
			\begin{equation}\label{s-q-x}s_*(q_x)=\ang{l,-lc,-lx,lbcx}.\end{equation}
   Since $x\in N_cN_{bc}$, there exist $w_1,w_2,w_3,w_4\in F$ such that \[w_1^2-cw_2^2=x(w_3^2-bcw_4^2)\neq 0.\] Multiplying both sides by $l$ and using (\ref{s-q-x}), we deduce that $s_*(q_x)$ is isotropic. It follows that $q_x$ represents a value in $F$. Since $q_x$ is anisotropic, it represents a value in $F^{\times}$, and so by (\ref{albert-cor-eq}) there exists $y\in F^\times$ such that $(\alpha x, \delta y)=0$ in $\on{Br}(F_a)$. This implies \Cref{albert-cor}(b) when $q_x$ is anisotropic, thus completing the proof.
		\end{rmk}

		\begin{lemma}\label{alpha-cup-c}
			Let $x\in F^\times$. Then $(\alpha x,c)=0$ in $\on{Br}(F_a)$ if and only if $x\in tN_cN_{ac}$.
		\end{lemma}
		
		\begin{proof}
			We have \[N_a(\alpha)=b(lv_2)^2=(1-c)(lv_2)^2=N_c(lv_2(1+\sqrt{c}))\] and \[\on{Tr}_a(\alpha)+\on{Tr}_c(lv_2(1+\sqrt{c}))=2lv_1+2lv_2=t.\] Now \Cref{comes-from-ac}(3) implies
			\[(\alpha,c)=(t,c)\qquad \text{in $\on{Br}(F_a)$},\]
			hence, adding $(x,c)$ to both sides,
			\[(\alpha x,c)=(tx,c)\qquad \text{in $\on{Br}(F_a)$}.\]
			Thus $(\alpha x,c)=0$ in $\on{Br}(F_a)$ if and only if $(tx,c)=0$ in $\on{Br}(F_a)$. By \cite[Chapter XIV, Proposition 2]{serre1979local}, this is in turn equivalent to the existence of $y\in F^\times$ such that $(tx,c)=(a,y)$ in $\on{Br}(F)$. By \Cref{chain-lemma}, this is equivalent to $tx\in N_cN_{ac}$, that is, $x\in tN_cN_{ac}$.
		\end{proof}
		
		\begin{proof}[Proof of \Cref{lambda-criterion}(b)]
   By \Cref{u5-cor}, the Massey product $\ang{a,b,c,a}$ vanishes over $F$ if and only if there exist $x,y\in F^\times$ such that \[(\alpha x,\delta y)=(\alpha x,c)=0 \qquad\text{in $\on{Br}(F_a)$.} \]
			By \Cref{albert-cor}(a) and \Cref{alpha-cup-c}, the two equations are satisfied if and only if $tN_cN_{ac}\cap N_cN_{bc}$ is non-empty, that is, $t\in N_cN_{ac}N_{bc}$.
		\end{proof}

		\subsection{Proof of Proposition \ref{q-t-s}}
		
		We maintain the notations and assumptions of the beginning of \Cref{massey-deg-section}.
		
		\begin{lemma}\label{sc+tb=pa}
			We have	$(r,a)+(s,b)+(t,c)=0$ in $\on{Br}(F)$. 
		\end{lemma}
		
		\begin{proof}
			We have
			\[u_1^2+v_2^2=(v_1^2-bv_2^2+cu_2^2)+v_2^2=v_1^2+c(v_2^2+u_2^2),\]
			hence
			\begin{align*}
				(v_1+u_1+v_2)^2&=v_1^2+(u_1^2+v_2^2)+2v_1u_1+2v_1v_2+2u_1v_2\\
				&=2v_1^2+c(v_2^2+u_2^2)+2v_1u_1+2v_1v_2+2u_1v_2\\
				&=t+c(v_2^2+u_2^2).
			\end{align*}
			In other terms, the conic of equation $tX^2+c(v_2^2+u_2^2)Y^2=Z^2$ (which is smooth if $v_2^2+u_2^2\neq 0$) has the $F$-point $(1:1:v_1+u_1+u_2)$. Thus by \Cref{cup-norm}
			\begin{equation}\label{t-cup-c}
				(t,c)=(t,v_2^2+u_2^2)\qquad \text{if $v_2^2+u_2^2\neq 0$},
			\end{equation}
			and $(t,c)=0$ if $v_2^2+u_2^2=0$. Similarly,
			\begin{equation}\label{s-cup-b}
				(s,b)=(s,v_2^2+u_2^2)\qquad \text{if $v_2^2+u_2^2\neq 0$},
			\end{equation}
			and $(s,b)=0$ if $v_2^2+u_2^2=0$. We also have
			\begin{align*}
    (v_2u_2+v_1u_2+v_2u_1)^2&=v_2^2u_2^2+v_1^2u_2^2+v_2^2u_1^2+ 2v_2u_2(v_1u_1+v_1u_2+v_2u_1)\\
				&=-v_2^2u_2^2+v_1^2u_2^2+v_2^2u_1^2+2v_2u_2(v_1u_1+v_1u_2+v_2u_1+v_2u_2)\\
				&=-v_2^2u_2^2+(a+bv_2^2)u_2^2+v_2^2(a+cu_2^2)+r\\
            &=a(v_2^2+u_2^2)+r.
			\end{align*}
			Now \Cref{cup-norm} implies
			\begin{equation}\label{p-cup-a}
				(r,a)=(r,v_2^2+u_2^2)\qquad \text{if $v_2^2+u_2^2\neq 0$},
			\end{equation}
			and $(r,a)=0$ if $v_2^2+u_2^2=0$. In particular, when $v_2^2+u_2^2=0$ we have $(r,a)=(s,b)=(t,c)=0$, which implies the conclusion in this case. Suppose now that $v_2^2+u_2^2\neq 0$. Note that
			\begin{equation}\label{note1}rst=2v_2u_2\cdot\left(\frac{rl}{v_2u_2}\right)^2.\end{equation} Finally, 
			\begin{equation}\label{note2} (2v_2u_2,v_2^2+u_2^2)=0\end{equation} since the smooth conic of equation $(2v_2u_2)X^2+(v_2^2+u_2^2)Y^2=Z^2$ has the $F$-point $(1:1:v_2+u_2)$. Putting 
			(\ref{t-cup-c})-(\ref{note2}) together, we conclude that
			\begin{align*}
				(r,a)+(s,b)+(t,c)&=(r,v_2^2+u_2^2)+(s,v_2^2+u_2^2)+(t,v_2^2+u_2^2)\\
                &=(rst, v_2^2+u_2^2)\\
				&=(2v_2u_2,v_2^2+u_2^2)\\
				&=0,
			\end{align*}
			which completes the proof.
		\end{proof}
		
		\begin{lemma}\label{formal}
			Let $a',b',c'\in F^\times$ be such that $(a',a)+(b',b)+(c',c)=0$ in $\on{Br}(F)$, then
			\begin{equation*}\label{a'b'c'}a'\in N_aN_{ab}N_{ac}\iff b'\in N_bN_{ab}N_{bc} \iff c'\in N_cN_{ac}N_{bc}.\end{equation*}
		\end{lemma}
		
		\begin{proof}
			Suppose that $c'=n_cn_{ac}n_{bc}$, where $n_c\in N_c$, $n_{ac}\in N_{ac}$ and $n_{bc}\in N_{bc}$. Then
			\begin{align*}
				0&=(c',c)+(b',b)+(a',a)\\
				&=(n_cn_{ac}n_{bc},c)+(b',b)+(a',a)\\
				&=(n_{ac},c)+(n_{bc},c)+(b',b)+(a',a)\\
				&=(n_{ac},a)+(n_{bc},b)+(b',b)+(a',a)\\
				&=(n_{ac}a',a)+(n_{bc}b',b).
			\end{align*}
			Thus \Cref{chain-lemma} implies that $n_{ac}a'\in N_aN_{ab}$ and $n_{bc}b'\in N_bN_{ab}$. Since the statement of \Cref{a'b'c'} is symmetric in $a,b,c$,  this completes the proof.
		\end{proof}
		
		\begin{proof}[Proof of \Cref{q-t-s}]
			Immediate from \Cref{sc+tb=pa} and \Cref{formal}.
		\end{proof}
		
		\subsection{Proof of Theorem \ref{massey-deg}}\label{conclude-proof}
		
		As anticipated,  \Cref{massey-deg} will follow from \Cref{lambda-criterion} and \Cref{q-t-s}.

		\begin{proof}[Proof of \Cref{massey-deg}]
			Suppose that the Massey product $\ang{a,b,c,a}$ is defined over $F$. By \Cref{cup-products-zero} we have $(a,b)=(b,c)=(c,a)=0$ in $\on{Br}(F)$. If $F$ is a finite field, $\ang{a,b,c,a}$ vanishes by \Cref{dwyer-pro-free}. We may thus assume that $F$ is infinite. The Massey product $\ang{a,b,c,a}$ depends only on the classes of $a,b,c$ in $F^\times/F^{\times 2}$. Since $(b,c)=0$, by \Cref{cup-norm} there exist $x,y\in F^\times$ such that $bx^2+cy^2=1$. Replacing $b$ by $bx^2$ and $c$ by $cy^2$, we may suppose that $b+c=1$.

			Consider $(\A^2_F)^2$ with coordinates $(v_1,v_2,u_1,u_2)$, and let $Y\subset (\A^2_F)^2$ be the locally-closed subvariety given by (\ref{vi-ui-def}).  Since $(a,b)=(a,c)=0$ in $\on{Br}(F)$ and $F$ is infinite, $Y$ is $F$-rational and so has an $F$-point, that is, (\ref{vi-ui-def}) has a solution. Now \Cref{lambda-criterion}(a) implies that $r\in N_{a}N_{ab}N_{ac}$. It follows from \Cref{q-t-s} that 
   $t\in N_{c}N_{ac}N_{bc}$. By \Cref{lambda-criterion}(b), the Massey product $\ang{a,b,c,a}$ vanishes, as desired.
		\end{proof}

  \begin{rmk}
The final part of the proof of \Cref{massey-deg} may be replaced by the following specialization argument. Suppose that $F$ is infinite, consider $(\A^2_F)^2$ with coordinates $(v_1,v_2,u_1,u_2)$, and let $Z\subset (\A^2_F)^2$ be the smooth variety given by the first two equalities of (\ref{vi-ui-def}).  Then the restrictions to $Z$ of the coordinate functions of $(\A^2_F)^2$ satisfy (\ref{vi-ui-def}) over $F(Z)$. Since $Z$ is smooth and has an $F$-point, by \Cref{specialize-variety} we may replace $F$ by $F(Z)$. The conclusion then follows from \Cref{lambda-criterion} and \Cref{q-t-s}. A similar argument could also be used for the proof of \Cref{massey-double-deg} in the next section.
  \end{rmk}

		\section{Proof of Theorem \ref{massey-double-deg}}\label{massey-double-deg-section}

  Let  $b,c\in F^\times$ such that $b+c=1$. In particular $(b,c)=0$ in $\on{Br}(F)$. Suppose further that $(bc,b)=(bc,c)=0$. This is equivalent to $(b,-1)=(c,-1)=0$, that is, by \Cref{cup-norm}, to $-1\in N_b\cap N_c$. 
		
		Before moving to the proof of \Cref{massey-double-deg}, we specialize some of the definitions of \Cref{massey-deg-section} to the case $a=d=bc$. Let $(v_1,v_2,u_1,u_2)\in (F^{\times})^4$ be a solution of (\ref{vi-ui-def}), where we set $a=d=bc$. Set
		\[v\coloneqq v_1+v_2\sqrt{b}\in F_b^{\times},\qquad u\coloneqq u_1+u_2\sqrt{c}\in F_c^{\times},\]
		so that $N_b(v)=N_{c}(u)=bc$. The definition of (\ref{define-l}) specializes to 
		\begin{equation}\label{define-l-special}l\coloneqq v_1+u_1,\qquad \alpha\coloneqq lv_1+l\sqrt{bc}\in F_{bc}^{\times},\qquad \delta\coloneqq lu_1+l\sqrt{bc}\in F_{bc}^{\times}.\end{equation}
  We have
		\[N_{bc}(\alpha)=b(lv_2)^2,\qquad N_{bc}(\delta)=c(lu_2)^2.\]
		We may also write \[N_{bc}(\alpha)=b(lv_2)^2=(1-c)(lv_2)^2=N_c(lv_2(1+\sqrt{c})).\]
		Following the definition given after (\ref{vi-ui-def}), we have \[t\coloneqq 2l(v_1+v_2) = \on{Tr}_{bc}(\alpha)+\on{Tr}_c(lv_2(1+\sqrt{c})).\]
		Since $v_1+v_2\neq 0$ by (\ref{vi-ui-def}), we have $t\neq 0$.
		\begin{lemma}\label{-d2}
			We have $-t^2\in N_{b,c}$.
		\end{lemma}
		
		\begin{proof}
			By \Cref{comes-from-ac}(2) applied to $\rho=\alpha$ and $\mu=lv_2(1+\sqrt{c})$, we have that $b(lv_2)^2t^2\in N_{b,c}$.	Thus, in order to prove that $-t^2\in N_{b,c}$, it suffices to show that $-b(lv_2)^2\in N_{b,c}$. Since $N_b(v)=N_c(u)$, by \Cref{comes-from-ac}(2) applied to $\rho=v$ and $\mu=u$, we have $4bcl^2\in N_{b,c}$.
   
   We also have $N_b(v/\sqrt{b})=-c=N_c(\sqrt{c})$, hence by \Cref{comes-from-ac}(2) applied to $\rho=v/\sqrt{b}$ and $\mu=\sqrt{c}$ we obtain that $-4cv_2^2\in N_{b,c}$.
		Since $16c^2=N_{b,c}(2\sqrt{c})$, it follows that
			\[-b(lv_2)^2=(4bcl^2)\cdot (-4cv_2^2)\cdot (16c^2)^{-1} \in N_{b,c},\]
			as desired.
		\end{proof}
		
		\begin{proof}[Proof of \Cref{massey-double-deg}]
			If $F$ is a finite field, then $N_{b,c}=F^\times$ and by \Cref{dwyer-pro-free} every Massey product over $F$ vanishes, hence \Cref{massey-double-deg} is true in this case. From now on, we suppose that $F$ is infinite. Multiplying $a,b,c,d$ by non-zero squares does not alter the Massey product $\ang{a,b,c,d}$, hence we may suppose that $a=d=bc$ and $b+c=1$. By \Cref{massey-deg}, we know that (1) is equivalent to (2). 
			
			We now prove that (2) is equivalent (3). We have $(b,c)=0$, and we may suppose that $(bc,b)=(bc,c)=0$, as it is implied by either (2) or (3). Consider $(\A^2_F)^2$ with coordinates $(v_1,v_2,u_1,u_2)$, and let $Y\subset (\A^2_F)^2$ be the locally-closed subvariety given by (\ref{vi-ui-def}). The $F$-variety $Y$ is rational because $(bc,b)=(bc,c)=0$. Thus, since $F$ is infinite, $Y$ has an $F$-point, that is, (\ref{vi-ui-def}) has a solution. 
   
			Recall that  we defined  $\alpha\in F_a^\times$ and $\delta\in F_d^\times$ such that $N_{bc}(\alpha)=b$ in $F^\times/F^{\times 2}$ and $N_{bc}(\delta)=c$ in $F^\times/F^{\times 2}$ in (\ref{define-l-special}), as a special case of (\ref{define-l}). By \Cref{u5-cor}, the Massey product $\ang{bc,b,c,bc}$ vanishes if and only if there exist $x,y\in F^\times$ such that 
			$(\alpha x,c)=(\alpha x,\delta y)=0$ in $\on{Br}(F_{bc})$.
   By \Cref{albert-cor} and \Cref{alpha-cup-c}, these equalities are equivalent to $N_{c}N_{bc}\cap tN_{b}N_{c}\neq\emptyset$. The latter is equivalent to $t\in N_{b}N_{c}N_{bc}$, which by \Cref{biquadratic-triple-norm}(2) is equivalent to $t^2\in N_{b,c}$. By \Cref{-d2}, this is equivalent to $-1\in N_{b,c}$, as desired. This shows that (2) is equivalent to (3), as desired. 
   \end{proof}

As an application of \Cref{massey-double-deg}, we recover the Harpaz--Wittenberg example \cite[Example A.15]{guillot2018fourfold}. 

Let $K/F$ be a Galois extension of number fields, $v$ be a place of $F$ and $w$ be a place of $K$ above $v$. Let $F_v$ be the completion of $F$ at $v$ and $K_w$ be the completion of $K$ at $w$. By definition, the local degree of $K/F$ at $v$ is equal to $[K_w:F_v]$.
        
        Let $w_1,\dots,w_m$ be the places of $K$ above $v$. Since $\on{Gal}(K/F)$ acts transitively on the $w_i$, the local degree of $v$ does not depend on the choice of $w$. Moreover, the natural homomorphism of $F_v$-algebras $K\otimes_FF_v\to \prod_{i=1}^mK_{w_i}$ is an isomorphism (see e.g. \cite[Chapter VII, Proof of Proposition 1.2]{cassels1967algebraic}), hence the local degree of $v$ is a divisor of $[K:F]$, and it is equal to $[K:F]$ if and only if $m=1$, that is, if and only if $K\otimes_FF_v$ is a field.

    \begin{lemma}\label{-1-not-norm}
    Let $F=\Q$. Then $-1$ does not belong to $N_{2,17}$.
    \end{lemma}

\begin{proof}
      Let $K\coloneqq \Q(\sqrt{2}, \sqrt{17})$. It is not difficult to prove that the local degree of $K/\Q$ at any place $v$ of $\Q$ is either $1$ or $2$. For all $c\in \Q^\times$, define
			\[\omega(c)\coloneqq \prod_{v\in S_1}(17,c)_v,\]
			where $S_1$ is the set of places $v$ of $\Q$ that split in $\Q(\sqrt{2})$, and $(17,c)_v$ denotes the symbol $(17,c)$ in $\on{Br}(\Q_v)$. By a result due to Serre and Tate, $c$ belongs to $N_2 N_{17} N_{34}$ if and only if $\omega(c)=1$; see \cite[Exercise 5.2]{cassels1967algebraic} or \cite[Lemma p. 114]{hurlimann1986h3}.  
   
   Note that $3$ does not belong to $S_1$ while $17$ belongs to $S_1$. Moreover, $3$ is not square modulo $17$, and hence $\omega(3) = (17,3)_{17} = -1$. By the aforementioned result of Serre and Tate, $3$ is not in $N_2 N_{17} N_{34}$, which by \Cref{biquadratic-triple-norm}(2) implies that $9$ does not belong to $N_{2,17}$. On the other hand, 
			\[-9=N_{2,17}\left(1-\frac{3}{2}\sqrt{2}-\frac{1}{2}\sqrt{34}\right),\] and hence $-1$ does not belong to $N_{2,17}$.
\end{proof}
  
		\begin{prop}[Harpaz--Wittenberg]\label{harpaz}
			Let $F=\Q$, $b=2$, $c=17$ and $a=d=bc=34$. Then $(a,b)=(b,c)=(c,d)=0$ but the Massey product $\ang{a,b,c,d}$ is not defined.
		\end{prop}
		
		\begin{proof}
        We first show that $(34,2)=(2,17)=(17,34)=0$, or equivalently $(2,17)=(2,-1)=(17,-1)=0$. These follow from the identities
        \[2\cdot 4^2+17\cdot 1^2=7^2,\qquad 2\cdot 1^2-1\cdot 1^2=1,\qquad 17\cdot 1^2 -1\cdot 4^2=1^2.\]
			By \Cref{massey-double-deg}, to prove that $\ang{a,b,c,d}$ is not defined it suffices to show that $-1$ does not belong to $N_{2,17}$, which was proved in \Cref{-1-not-norm}.
		\end{proof}

		\section{Proof of Theorem \ref{mainthm-positselski}}\label{positselski-sec}

Recall that a field $E$ is said to be \emph{$2$-special} if the degree of every finite field extension of $E$ is a power of $2$.
		
		\begin{lemma}\label{sum-is-even}
			Suppose that $F$ is a $2$-special field, and let $a\in F^{\times}\setminus F^{\times 2}$. Let $X$ be a non-split smooth projective conic over $F$, let $\phi\colon X_{F_a}\to X$ be the projection map, and consider the proper pushforward homomorphism \[\phi_*\colon \on{Div}(X_{F_a})\to \on{Div}(X).\] Let $\sum m_xx\in \on{Div}(X)$ be a principal divisor in the image of $\phi_*$. Then the sum of the $m_x$ over all closed points $x\in X$ such that $F(x)\simeq F_a$ is even.
		\end{lemma}

  		\begin{proof}
   Let $x\in X$ be a closed point. Note that $\on{deg}(x)\coloneqq [F(x):F]$ is a power of $2$, and it is different from $1$ since $X(F)=\emptyset$. Moreover, $\phi^{-1}(x)=\on{Spec}(F_a\otimes_FF(x))$ and $[F_a\otimes_FF(x):F(x)]=2$, hence either $F_a\otimes_FF(x)\simeq F_a\times F_a$ or it is a quadratic field extension of $F(x)$. In other words, either $\phi^{-1}(x)$ is the union of two closed points of degree $1$ over $x$, or it consists of a single closed point of degree $2$ over $x$. Therefore, there are three mutually exclusive possibilities for the closed point $x$:
   \begin{enumerate}[label=(\roman*)]
       \item $\on{deg}(x)$ is divisible by $4$,
       \item $\on{deg}(x)=2$ and $\phi^{-1}(x)$ consists of a single closed point of degree $2$ over $x$,
       \item $\on{deg}(x)=2$ and $\phi^{-1}(x)$ consists of two closed points of degree $1$ over $x$.
   \end{enumerate}
   Moreover, $x$ satisfies (iii) if and only if $\on{deg}(x)=2$ and $a\in F(x)^{\times 2}$, that is, if and only if $F(x)\simeq F_a$. Thus, the proof will be complete once we show that the sum of the $m_x$ over all closed points $x\in X$ satisfying (iii) is divisible by $4$.

   If $x$ satisfies (i), the number $m_x\on{deg}(x)$ is divisible by $4$. If $x$ satisfies (ii), then $m_x$ is divisible by $2$, hence $m_x\on{deg}(x)$ is divisible by $4$. Since the divisor $\sum m_xx$ is principal, its degree is equal to zero, that is, $\sum m_x\on{deg}(x)=0$. It follows that the sum of the $m_xx$ for those $x$ as in case (iii) is divisible by $4$, as desired.
		\end{proof}

		\begin{lemma}\label{example-not-defined}
			Let $E$ be a $2$-special field, and let $a,b\in E^{\times}$ be such that 
			\begin{enumerate}
				\item $a$, $b$ and $c\coloneqq 1-b$ are independent in $E^{\times}/E^{\times 2}$, and
				\item $(a,c)=0$ and $(a,b)\neq 0$ in $\on{Br}(E)$.
			\end{enumerate}
			Let $X$ be the smooth projective conic over $E$ corresponding to $(a,b)$, and set $F\coloneqq E(X)$.
			Then $(a,b)=(a,c)=(b,c)=0$ in $\on{Br}(F)$ but the Massey product $\ang{a,b,c,a}$ is not defined over $F$.
		\end{lemma}

		\begin{proof}
			We have $(a,b)=(b,c)=(c,a)=0$ in $\on{Br}(F)$ because $b+c=1$ and every quaternion algebra splits over the function field of the corresponding conic.
			
			Consider the projective plane $\P^2_E$ with homogeneous coordinates $v_0,v_1,v_2$, and choose a model of $X\subset \P^2_E$ given by the equation \[v_1^2-bv_2^2=av_0^2.\]
			Define \[f\coloneqq \frac{v_1+v_2}{v_2}\in F^{\times}.\]
			Simple computations show that the equation $v_1+v_2=0$ cuts out a point $x_1\in X$ of degree $2$ with residue field $E(x_1)=E_{ac}$, and that the equation $v_2=0$ cuts out a point $x_2\in X$ of degree $2$ with $E(x_2)=E_a$. Thus
			\[\on{div}(f)=x_1-x_2.\]
			Since $(a,b)\neq 0$ in $\on{Br}(E)$, the field $E$ must be infinite. Thus, as $(a,c)=0$ in $\on{Br}(E)$, we may find $u_1,u_2\in E^\times$ such that $u_1+u_2\neq 0$ and $u_1^2-cu_2^2=a$. Suppose by contradiction that the Massey product $\ang{a,b,c,a}$ is defined. Then by \Cref{lambda-criterion}(a) we have
			\[2(u_1+u_2)u_2f=2(v_1+v_2)(u_1+u_2)v_2u_2v_2^{-2}\in N_aN_{ab}N_{ac}.\]
			Since $2(u_1+u_2)u_2\in E^{\times}$, we may write
			\[f=f_0n_an_{ab}n_{ac}\]
			for some $f_0\in E^{\times}$, $n_a\in N_a$, $n_{ab}\in N_{ab}$ and $n_{ac}\in N_{ac}$. Passing to divisors and using that $\on{div}(f_0)=0$, we conclude that
			\begin{equation}\label{div-f-norms}\on{div}(f)=\on{div}(n_a)+\on{div}(n_{ab})+\on{div}(n_{ac}).\end{equation}
			We write $\on{div}(n_a)=\sum m_x x$, $\on{div}(n_{ab})=\sum m'_x x$ and $\on{div}(n_{ac})=\sum m''_x x$. By \Cref{sum-is-even}, the sum of the $m_x$ over all closed points $x$ such that $E(x)\simeq E_a$ is even. By assumption $ab$ and $ac$ are not squares in $E$. It follows that the inverse image of $X_{ab}\to X$ and $X_{ac}\to X$ at every closed point of $x$ with residue field $E(x)\simeq E_a$ consists of a single closed point whose residue field has degree $2$ over $E(x)$. Therefore the sum of the $m'_x$ (resp. $m''_x$) over all closed points $x$ such that $E(x)\simeq E_a$ is also even. However, by (\ref{div-f-norms}) the sum of the $m_x+m'_x+m''_x$ over all closed points $x$ such that $E(x)\simeq E_a$ is equal to $1$, which is a contradiction. Therefore the Massey product $\ang{a,b,c,a}$ is not defined over $F$.
		\end{proof}
  
      Recall from the Introduction that the DGA $C^{\smallbullet}(F,\Z/2\Z)$ is formal if it is quasi-isomorphic to its cohomology algebra $H^{\smallbullet}(F,\Z/2\Z)$, viewed as a DGA with zero differential. If $C^{\smallbullet}(F,\Z/2\Z)$ is formal, then by \cite[Theorem 3.8]{pal2022real} for all $n\geq 3$ and all $a_1,\dots,a_n\in F^\times$ such that $a_1\cup a_2=a_2\cup a_3=\dots =a_{n-1}\cup a_n=0$, the Massey product $\ang{a_1,\dots,a_n}$ vanishes. Therefore the next theorem, which implies \Cref{mainthm-positselski}, answers Positselski's \Cref{formal-question} affirmatively.
		
		\begin{thm}\label{counter-quadrelli}
			Let $F_0$ be a field of characteristic different from $2$. 
			
			(a) There exist a field extension $F/F_0$ and elements $a,b,c\in F^\times$, independent in $F^\times/F^{\times 2}$, such that $(a,b)=(b,c)=(a,c)=0$ in $\on{Br}(F)$ but the Massey product $\ang{a,b,c,a}$ is not defined over $F$.
			
			(b) There exist a field extension $F/F_0$ and elements $a,b,c,d\in F^\times$, independent in $F^\times/F^{\times 2}$, such that $(a,b)=(b,c)=(a,d)=0$ in $\on{Br}(F)$ but the Massey product $\ang{a,b,c,d}$ is not defined over $F$.
		\end{thm}
		
		\begin{proof}
			(a) Let $a$ and $b$ be algebraically independent variables over $F_0$, set $F_1\coloneqq F_0(a,b)$, and define $c\coloneqq 1-b\in F_1$. Write $C$ for the smooth projective conic over $F_1$ corresponding to $(a,c)\in \on{Br}(F_1)$, let $F_2 \coloneqq F_1(C)$ and let $F_3$ be a $2$-closure of $F_2$, that is, the subfield of $(F_2)_{\on{sep}}$ fixed by a Sylow $2$-subgroup of $\Gamma_{F_2}$. The field $F_3$ is $2$-special. We have the inclusions
			\[F_0\subset F_1\subset F_2\subset F_3.\]
			In order to complete the proof, it suffices to show that assumptions (1) and (2) of \Cref{example-not-defined} are satisfied by $a,b,c$ over $E=F_3$.
			
			(1) Consider the group homomorphisms
			\[F_1^\times/F_1^{\times 2}\to F_2^\times/F_2^{\times 2}\to F_3^\times/F_3^{\times 2}.\]
			The homomorphism on the left is injective because $F_1$ is algebraically closed in $F_2$, and the homomorphism on the right is injective by a restriction-corestriction argument.
			It is clear that $a,b,c$ are independent in $F_1^\times/F_1^{\times 2}$, hence they are independent in  $F_3^\times/F_3^{\times 2}$.
			
			(2) We have $(a,c)=0$ in $\on{Br}(F_2)$ because $C$ is the conic corresponding to $(a,c)$, hence $(a,c)=0$ in $\on{Br}(F_3)$. Suppose that $(a,b)=0$ in $\on{Br}(F_3)$. Then there exists a finite extension $L/F_2$ of odd degree such that $(a,b)=0$ in $\on{Br}(L)$. Since $[L:F_2]$ is odd, a restriction-corestriction argument shows that the restriction map $\on{Br}(F_2)\to \on{Br}(L)$ is injective, hence $(a,b)=0$ in $\on{Br}(F_2)$. By \cite[Proposition 7.2.4(b)]{colliot2021brauer}, the kernel of the restriction map $\on{Br}(F_1)\to \on{Br}(F_2)$ is generated by $(a,c)$, hence either $(a,b)=0$ or $(a,b)=(a,c)$ (that is, $(a,b(1-b))=0$) in $\on{Br}(F_1)$. Taking residues with respect to the valuation determined by $a$, we see that neither of these equalities can be true, a contradiction. Therefore $(a,b)\neq 0$ in $\on{Br}(F_3)$, as desired.
			
			(b) By (a), there exist a field extension $L/F_0$ and $a,b,c\in L^\times$, independent in $L^\times/L^{\times 2}$, such that $(a,b)=(b,c)=(c,a)=0$ in $\on{Br}(L)$ and the Massey product $\ang{a,b,c,a}$ is not defined. Let $F\coloneqq L(u)$, where $u$ is a variable over $L$, let $R\coloneqq L[u]_{(u-1)}\subset F$, and define $d\coloneqq ua\in L^\times$. Then $a,b,c,d$ belong to $R^\times$, they are independent in $L^\times/L^{\times 2}$, and by \Cref{specialize} the Massey product $\ang{a,b,c,d}$ is not defined over $F$. 
		\end{proof}
 
		\appendix
		
		\section{Lemmas on biquadratic extensions}\label{appendix}
		
		We collect some known results on Galois $(\Z/2\Z)^2$-algebras that are needed for the proofs of Theorems \ref{massey-deg} and \ref{massey-double-deg}. They are all consequences of Hilbert's Theorem~90.
		
		Let $F$ be a field of characteristic different from $2$, let $a,b\in F^\times$, and let \[F_{a,b}\coloneqq F[x_a,x_b]/(x_a^2-a,x_b^2-b)\] be the corresponding \'etale $F$-algebra. We write $(\Z/2\Z)^2=\ang{\sigma_a,\sigma_b}$ and view $F_{a,b}$ as a Galois $(\Z/2\Z)^2$-algebra via 
		\[\sigma_a(x_a)=-x_a,\quad \sigma_a(x_b)=x_b,\qquad \sigma_b(x_a)=x_a,\quad \sigma_b(x_b)=-x_b.\]

		\begin{lemma}\label{chain-lemma}
			Let $u,v\in F^{\times}$. Then $(a,u)=(b,v)$ in $\on{Br}(F)$ if and only if there exist $n_a\in N_a$, $n_b\in N_b$ and $n_{ab}\in N_{ab}$ such that $u=n_an_{ab}$ and $v=n_bn_{ab}$.
		\end{lemma}
		
		\begin{proof}
			It follows from \Cref{cup-norm} that $(a,n_an_{ab})=(b,n_bn_{ab})$ for all $n_a\in N_a$, $n_b\in N_b$ and $n_{ab}\in N_{ab}$. Conversely, if $(a,u)=(b,v)$ then, by the Common Slot Theorem \cite[Chapter III, Theorem 4.13]{lam2005introduction}, there exists $w\in F^\times$ satisfying
			\[(a,u)=(a,w)=(b,w)=(b,v).\]
			It follows from \Cref{cup-norm} that $w\in N_{ab}$, $u/w\in N_a$ and $v/w\in N_b$, as desired.
		\end{proof}

		\begin{lemma}\label{product-a-b}
			For all $\omega\in F_{a,b}^\times$, we have $(\sigma_a-1)(\sigma_b-1)\omega=1$ if and only if $\omega=\omega_a\omega_b$ for some $\omega_a\in F_a^\times$ and $\omega_b\in F_b^\times$.
		\end{lemma}
		
		\begin{proof}
			See \cite[Theorem 4]{dwilewicz2007hilbert}. Let $T\subset R_{a,b}(\G_{\on{m}})$ be the torus image of the multiplication map \[\mu\colon R_a(\G_{\on{m}})\times R_b(\G_{\on{m}})\to R_{a,b}(\G_{\on{m}}).\] A character lattice computation shows that the torus $T$ is defined by the equation $(\sigma_a-1)(\sigma_b-1)\omega=1$ inside $R_{a,b}(\G_{\on{m}})$. Therefore, for all $\omega\in F_{a,b}^\times$, we have $\omega\in T(F)$ if and only if $(\sigma_a-1)(\sigma_b-1)\omega=1$. We have a short exact sequence
			\[1\to \G_{\on{m}}\to R_a(\G_{\on{m}})\times R_b(\G_{\on{m}})\xrightarrow{\mu} T\to 1.\]
			Passing to $F$-points, we see that an element $\omega\in F_{a,b}^\times$ belongs to $T(F)$ if and only if $\omega=\omega_a\omega_b$ for some $\omega_a\in F_a^\times$ and $\omega_b\in F_b^\times$.
		\end{proof}
		
		\begin{lemma}\label{biquadratic-triple-norm}
	(1) Let $\rho \in F_{a,b}^\times$. Then $N_{a,b}(\rho)\in F^{\times 2}$ if and only if $\rho\in F_a^\times F_b^\times F_{ab}^\times$.		
   
   (2) Let $u\in F^{\times}$. Then $u\in N_aN_{b}N_{ab}$ if and only if $u^2\in N_{a,b}$.
		\end{lemma}
		
		\begin{proof}
    (1) Suppose first that $\rho=\rho_a\rho_b\rho_{ab}$, where $\rho_a\in F_a^\times$, $\rho_b\in F_b^\times$ and $\rho_{ab}\in F_{ab}^\times$. Then 
    \[N_{a,b}(\rho)=N_{a,b}(\rho_a\rho_b\rho_{ab})=N_a(\rho_a)^2N_b(\rho_b)^2N_{ab}(\rho_{ab})^2\in F^{\times 2}.\]
    Conversely, suppose that $N_{a,b}(\rho)=x^2$ for some $x\in F^\times$. Let \[\omega\coloneqq N_{F_{a,b}/F_{ab}}(\rho)/x \in F_{ab}^\times.\] Then $N_{ab}(\omega)=N_{a,b}(\rho)/x^2=1$, hence by Hilbert's Theorem~90 there is $\rho_{ab}\in F_{ab}^\times$ such that $(1-\sigma_a)\rho_{ab}=\omega$. Since $\sigma_a\sigma_b$ fixes $F_{ab}$, we have 
			\[(\sigma_a-1)(\sigma_b-1)\rho_{ab}=(2-2\sigma_a)\rho_{ab}=\omega^2,\]
   where the first equality follows from the fact that $\sigma_a=\sigma_b$ on $F_{ab}$. On the other hand,
			\[(\sigma_a-1)(\sigma_b-1)\rho=N_{F_{a,b}/F_{ab}}(\rho)^2/N_{a,b}(\rho)=\omega^2,\]
			and hence \[(\sigma_a-1)(\sigma_b-1)(\rho/\rho_{ab})=1.\]
			By \Cref{product-a-b}, we deduce that $\rho=\rho_{ab}\rho_a\rho_b$ for some $\rho_a\in F_a^\times$ and $\rho_b\in F_b^\times$.
  
	(2)		See \cite[Exercise 5.1]{cassels1967algebraic}. Suppose that $u=N_a(\rho_a)N_b(\rho_b)N_{ab}(\rho_{ab})$, where $\rho_a\in F_a^\times$, $\rho_b\in F_b^\times$ and $\rho_{ab}\in F_{ab}^\times$. Then 
			\[N_{a,b}(\rho_a\rho_b\rho_{ab})=N_a(\rho_a)^2N_b(\rho_b)^2N_{ab}(\rho_{ab})^2=u^2.\]
	Conversely, suppose that $u^2=N_{a,b}(\rho)$ for some $\rho\in F_{a,b}^\times$. By (1), there exist $\rho_a\in F_a^\times$, $\rho_b\in F_b^\times$ and $\rho_{ab}\in F_{ab}^\times$ such that $\rho=\rho_a\rho_b\rho_{ab}$. It follows that
    \[u^2=N_{a,b}(\rho)=N_a(\rho_a)^2N_b(\rho_b)^2N_{ab}(\rho_{ab})^2,\]
    hence either $u=N_a(\rho_a)N_b(\rho_b)N_{ab}(\rho_{ab})$ or $u=-N_a(\rho_a)N_b(\rho_b)N_{ab}(\rho_{ab})$. Since $-a\in N_a$, $-b\in N_b$ and $-ab\in N_{ab}$, we have $-1\in N_aN_bN_{ab}$, hence $u\in N_aN_bN_{ab}$ in either case.
		\end{proof}

		\begin{lemma}\label{comes-from-ac}
			Let $\rho\in F_a^\times$ and $\mu\in F^\times_b$ be such that $N_a(\rho)=N_b(\mu)$. Set $d\coloneqq \on{Tr}_a(\rho)+\on{Tr}_b(\mu)$. Suppose that $d\neq 0$. Then:
			
			(1) $\mu d = N_a(\rho+\mu)$,
			
			(2) $N_b(\mu)d^2\in N_{a,b}$, and
			
			(3) $(\mu,a)=(d,a)$ in $\on{Br}(F_b)$.
		\end{lemma}
		
		\begin{proof}
			We have
        \begin{align*}
            N_a(\rho+\mu)&=(\rho+\mu)(\sigma_a(\rho)+\mu)\\
            &=\rho\sigma_a(\rho)+\rho\mu+\mu\sigma_a(\rho)+\mu^2\\
            &=\mu\sigma_b(\mu)+\rho\mu+\mu\sigma_a(\rho)+\mu^2\\
            &=\mu(\on{Tr}_a(\rho)+\on{Tr}_b(\mu))\\
            &=\mu d.
        \end{align*}
			This proves (1). Taking norms in (1) yields
			\[N_{a,b}(\rho+\mu)=N_b(\mu d)=N_b(\mu)d^2,\]
			which implies (2). Now \Cref{cup-norm} implies that $(\mu d,a)=0$, which is equivalent to (3). 
		\end{proof}

        \begin{lemma}\label{comes-from-ac-cor}
        We have $N_a\cap N_b=N_{a,b}F^{\times 2}$.
        \end{lemma}

    \begin{proof}
    If $F$ is finite, then $N_a=N_b=N_{a,b}=F^\times$, which implies the conclusion. We may thus assume that $F$ is infinite. 
    
    Let $u\in N_{a,b}F^{\times 2}$, and let  $\epsilon\in F_{a,b}^\times$ and $v\in F^\times$ be such that $u=N_{a,b}(\epsilon)v^2$. Then $u=N_a(N_b(\epsilon)v)\in N_a$, and $u=N_b(N_a(\epsilon)v)\in N_b$, hence $u\in N_a\cap N_b$.

    Conversely, let $u\in N_a\cap N_b$. Let $\mu \in F_b^\times$ such that $N_b(\mu)=u$. The solutions $\rho\in F_a^\times$ to $N_a(\rho)=u$ form the set of $F$-points of a smooth affine $F$-conic. Thus, since $F$ is infinite, there exists $\rho \in F_a^\times$ such that $N_a(\rho)=u$ and $\on{Tr}_a(\rho)\neq 0$. Therefore, replacing $\rho$ by $-\rho$ if necessary, we may suppose that $d\coloneqq \on{Tr}_a(\rho)+\on{Tr}_b(\mu)$ is non-zero. By \Cref{comes-from-ac}(2), we have $ud^2\in N_{a,b}$, as desired.
    \end{proof}

\section*{Acknowledgements}
We are grateful to Olivier Wittenberg for many helpful comments on an earlier version of this paper. We thank the anonymous referees for carefully reading our paper and for providing a number of useful comments and suggestions.


\begin{thebibliography}{KMRT98}

\bibitem[CF67]{cassels1967algebraic}
J.~W.~S. Cassels and A.~Fr\"{o}hlich, editors.
\newblock {\em Algebraic number theory}. Academic Press, London; Thompson Book
  Co., Inc., Washington, D.C., 1967.

\bibitem[CTS21]{colliot2021brauer}
Jean-Louis Colliot-Th\'{e}l\`ene and Alexei~N. Skorobogatov.
\newblock {\em The {B}rauer-{G}rothendieck group}, volume~71 of {\em Ergebnisse
  der Mathematik und ihrer Grenzgebiete. 3. Folge. A Series of Modern Surveys
  in Mathematics [Results in Mathematics and Related Areas. 3rd Series. A
  Series of Modern Surveys in Mathematics]}.
\newblock Springer, Cham, 2021.

\bibitem[DMSS07]{dwilewicz2007hilbert}
Roman Dwilewicz, Jan Minac, Andrew Schultz, and John Swallow.
\newblock Hilbert 90 for biquadratic extensions.
\newblock {\em The American Mathematical Monthly}, 114(7):577--587, 2007.

\bibitem[Dwy75]{dwyer1975homology}
William~G. Dwyer.
\newblock Homology, {M}assey products and maps between groups.
\newblock {\em J. Pure Appl. Algebra}, 6(2):177--190, 1975.

\bibitem[EM17]{efrat2017triple}
Ido Efrat and Eliyahu Matzri.
\newblock Triple {M}assey products and absolute {G}alois groups.
\newblock {\em J. Eur. Math. Soc. (JEMS)}, 19(12):3629--3640, 2017.

\bibitem[GMT18]{guillot2018fourfold}
Pierre Guillot, J\'{a}n Min\'{a}\v{c}, and Adam Topaz.
\newblock Four-fold {M}assey products in {G}alois cohomology.
\newblock {\em Compos. Math.}, 154(9):1921--1959, 2018.
\newblock With an appendix by Olivier Wittenberg.

\bibitem[GS17]{gille2017central}
Philippe Gille and Tam\'{a}s Szamuely.
\newblock {\em Central simple algebras and {G}alois cohomology}, volume 165 of
  {\em Cambridge Studies in Advanced Mathematics}.
\newblock Cambridge University Press, Cambridge, 2017.

\bibitem[H\"ur86]{hurlimann1986h3}
W.~H\"{u}rlimann.
\newblock {$H^3$} and rational points on biquadratic bicyclic norm forms.
\newblock {\em Arch. Math. (Basel)}, 47(2):113--116, 1986.

\bibitem[HW15]{hopkins2015splitting}
Michael~J. Hopkins and Kirsten~G. Wickelgren.
\newblock Splitting varieties for triple {M}assey products.
\newblock {\em J. Pure Appl. Algebra}, 219(5):1304--1319, 2015.

\bibitem[HW23]{harpaz2019massey}
Yonatan Harpaz and Olivier Wittenberg.
\newblock The {M}assey vanishing conjecture for number fields.
\newblock {\em  Duke Math. J.}, 172 (2023), no. 1, 1--41. 

\bibitem[KMRT98]{knus1998book}
Max-Albert Knus, Alexander Merkurjev, Markus Rost, and Jean-Pierre Tignol.
\newblock {\em The book of involutions}, volume~44 of {\em American
  Mathematical Society Colloquium Publications}.
\newblock American Mathematical Society, Providence, RI, 1998.
\newblock With a preface in French by J. Tits.

\bibitem[Lam05]{lam2005introduction}
T.~Y. Lam.
\newblock {\em Introduction to quadratic forms over fields}, volume~67 of {\em
  Graduate Studies in Mathematics}.
\newblock American Mathematical Society, Providence, RI, 2005.

\bibitem[Mat14]{matzri2014triple}
Eliyahu Matzri.
\newblock Triple {M}assey products in {G}alois cohomology.
\newblock {\em  arXiv:1411.4146}, 2014.

\bibitem[Mil80]{milne1980etale}
James S. Milne.
\newblock {\em \'Etale cohomology}, Princeton Mathematical Series, No. 33.
\newblock Princeton University Press, Princeton, N.J., 1980.

\bibitem[MT15a]{minac2015kernel}
J\'{a}n Min\'{a}\v{c} and Nguyen~Duy T\^{a}n.
\newblock The kernel unipotent conjecture and the vanishing of {M}assey
  products for odd rigid fields.
\newblock {\em Adv. Math.}, 273:242--270, 2015.

\bibitem[MT15b]{minac2015triple}
J\'{a}n Min\'{a}\v{c} and Nguyen~Duy T\^{a}n.
\newblock Triple {M}assey products over global fields.
\newblock {\em Doc. Math.}, 20:1467--1480, 2015.

\bibitem[MT16]{minac2016triple}
J\'{a}n Min\'{a}\v{c} and Nguyen~Duy T\^{a}n.
\newblock Triple {M}assey products vanish over all fields.
\newblock {\em J. Lond. Math. Soc. (2)}, 94(3):909--932, 2016.

\bibitem[MT17a]{minac2017counting}
J\'{a}n Min\'{a}\v{c} and Nguyen~Duy T\^{a}n.
\newblock Counting {G}alois {$\mathbb{U}_4(\mathbb{F}_p)$}-extensions using
  {M}assey products.
\newblock {\em J. Number Theory}, 176:76--112, 2017.

\bibitem[MT17b]{minac2017triple}
J\'{a}n Min\'{a}\v{c} and Nguyen~Duy T\^{a}n.
\newblock Triple {M}assey products and {G}alois theory.
\newblock {\em J. Eur. Math. Soc. (JEMS)}, 19(1):255--284, 2017.

\bibitem[PQ22]{pal2022real}
Ambrus P\'al and Gereon Quick. 
\newblock Real projective groups are formal. 
\newblock {\em arXiv:2206.14645} (2022).

\bibitem[PSc22]{pal2022brauer} 
Ambrus P\'al and Tomer M. Schlank.
\newblock Brauer-Manin obstruction to the local-global principle for the embedding problem. 
\newblock {\em Int. J. Number Theory}, 18 (2022), no. 7, 1535--1565.

\bibitem[PSz18]{pal2018strong}
Ambrus P\'al and Endre Szab\'o. 
\newblock The strong Massey vanishing conjecture for fields with virtual cohomological dimension at most 1.
\newblock {\em arXiv:1811.06192} (2018).

\bibitem[Pos17]{positselski2017koszulity}
Leonid Positselski.
\newblock Koszulity of cohomology = {$K(\pi,1)$}-ness + quasi-formality.
\newblock {\em J. Algebra}, 483:188--229, 2017.

\bibitem[Qua22]{quadrelli2022massey}
Claudio Quadrelli.
\newblock Massey products in {G}alois cohomology and the {E}lementary {T}ype
  {C}onjecture.
\newblock {\em arXiv:2203.16232}, 2022.

\bibitem[Ser79]{serre1979local}
Jean-Pierre Serre.
\newblock {\em Local fields}, volume~67 of {\em Graduate Texts in Mathematics}.
\newblock Springer-Verlag, New York-Berlin, 1979.
\newblock Translated from the French by Marvin Jay Greenberg.

\bibitem[Ser02]{serre2002galois}
Jean-Pierre Serre.
\newblock {\em Galois cohomology.} 
\newblock Springer Monographs in Mathematics. Springer-Verlag, Berlin, 2002.
\newblock Translated from the French by Patrick Ion and revised by the author. Corrected reprint of the 1997 English edition. 

\end{thebibliography}
\end{document}